\documentclass{elsarticle}[12pt]

\usepackage{amsmath}
\usepackage{amsfonts}
\usepackage{mathtools}
\usepackage{amssymb}
\usepackage{amsthm}
\usepackage{mathrsfs}
\usepackage{xcolor}
\usepackage{enumitem}
\usepackage{geometry}
\usepackage{hyperref}
\usepackage[nameinlink]{cleveref}
\usepackage{framed}
\usepackage{bm}
\usepackage{bbm}
\usepackage{natbib}
\setcitestyle{square,sort,comma,numbers}

\newtheorem{theorem}{Theorem}[section]

\newtheorem{lemma}[theorem]{Lemma}
\newtheorem{corollary}[theorem]{Corollary}
\newtheorem{conjecture}[theorem]{Conjecture}

\theoremstyle{remark}
\newtheorem{remark}[theorem]{Remark} 
\theoremstyle{definition}

\usepackage{xpatch}
\xpatchcmd{\proof}{\itshape}{\normalfont\proofnamefont}{}{}

\newcommand{\proofnamefont}{}

\renewcommand{\proofnamefont}{\bfseries}

\newcommand{\C}{\ensuremath{\mathbb{C}}}

\newcommand{\F}{\ensuremath{\mathbb{F}}}

\newcommand{\N}{\ensuremath{\mathbb{N}}}

\newcommand{\R}{\ensuremath{\mathbb{R}}}

\newcommand{\W}{\ensuremath{\mathscr{W}}}

\newcommand{\Ce}{\ensuremath{\mathscr{C}}}

\newcommand{\Le}{\ensuremath{\mathscr{L}}}
\newcommand{\Pe}{\ensuremath{\mathscr{P}}}
\newcommand{\Se}{\ensuremath{\mathscr{S}}}

\newcommand{\one}{\ensuremath{\mathbbm{1}}}

\DeclareMathOperator{\vect}{vec}

\DeclareMathOperator{\E}{E}
\DeclareMathOperator{\V}{V}
\DeclareMathOperator{\m}{m}

\DeclareMathOperator{\rk}{rk}
\DeclareMathOperator{\tr}{tr}

\DeclareMathOperator{\chr}{chr}

\newcommand{\bangle}[1]{\left\langle #1 \right\rangle}
\newcommand{\inprod}[2]{\bangle{#1, #2}}

\title{Equiangular lines via matrix projection}
\author[1]{Igor Balla \fnref{fn1}}
\ead{iballa1990@gmail.com}
\affiliation[1]{
organization={Einstein Institute of Mathematics},
addressline={Hebrew University of Jerusalem},
city={Jerusalem},
country={Israel}
}
\fntext[fn1]{Current affiliation: Leipzig University, Leipzig, Germany. Research supported in part by SNSF Project P2EZP2\_184522.}

\begin{document}
\begin{abstract}
In 1973, Lemmens and Seidel posed the problem of determining the maximum number of equiangular lines in $\R^r$ with angle $\arccos(\alpha)$ and provided a good partial answer in the regime $r \leq 1/\alpha^2 - 2$. At the other extreme where $r$ is at least exponential in $1/\alpha^2$, recent breakthroughs have led to an almost complete resolution of this problem. In this paper, we introduce a new method for obtaining upper bounds which unifies and improves upon previous approaches, thereby yielding bounds which bridge the gap between the aforementioned regimes and are best possible either exactly or up to a factor of two. Our approach relies on orthogonal projection of matrices with respect to the Frobenius inner product and as a byproduct, it yields the first extension of the Alon--Boppana theorem to dense graphs, with equality for strongly regular graphs corresponding to families of $\binom{r+1}{2}$ equiangular lines in $\R^r$. Applications of our method in the complex setting will be discussed as well.
\end{abstract}


\maketitle

\section{Introduction}

Given $n$ lines $l_1, \ldots, l_n$ passing through the origin in $\R^r$, we say that they are \emph{equiangular} if there exists a \emph{common angle} $\theta \in (0, \pi/2]$ such that for all $i \neq j$, the acute angle between $l_i$ and $l_j$ is $\theta$. If we choose a unit vector $v_i$ along each line $l_i$, an equivalent definition is that $|\inprod{v_i}{v_j}| = \cos(\theta)$ for all $i \neq j$, and this latter definition extends to complex lines in $\C^r$. Large sets of equiangular lines have been studied for over 75 years, arising naturally in a wide variety of different areas including elliptic geometry \cite{HS47, B53}, the theory of polytopes \cite{C73}, frame theory \cite{HS03, CKP13, HC17}, the theory of Banach spaces \cite{B19, KLL83, FS17, DL23}, and perhaps more surprisingly, quantum information theory \cite{Z11, RBSC04, FHS17, FS19, CFS02, FS03} with connections to algebraic number theory and Hilbert's twelfth problem \cite{B17, AFMY17, AFMY20, ABGHM22}. Moreover, determining the maximum number of equiangular lines in $\R^r$ is considered to be one of the founding problems of algebraic graph theory \cite{GR01} and it is well known due to the absolute bound (see \cite{LS73}) that the answer cannot be larger than $\binom{r+1}{2}$, which is tight up to a multiplicative constant.

In this paper, we will focus on the more refined question of determining $N^\R_{\alpha}(r)$, the maximum number of equiangular lines in $\R^r$ with common angle $\arccos(\alpha)$. This problem was first posed by Lemmens and Seidel \cite{LS73} in 1973, who gave a good partial answer when $r \leq 1/\alpha^2 - 2$ by proving the relative bound, for which many tight constructions are known \cite{FM16}. On the other hand, prior to this work, our understanding of the complementary regime $r \geq 1/\alpha^2 - 2$ has been much more limited. Recently, several breakthroughs have led to an almost complete solution of this problem when $r$ is very large relative to $1/\alpha$. The first result in this direction was due to Bukh \cite{B16}, which we subsequently improved in a joint work with Dr\"{a}xler, Keevash, and Sudakov \cite{BDKS18}, in particular showing that $N^\R_{\alpha}(r) \leq 2r-2$ when $r$ is at least exponential in $1/\alpha^2$. Jiang and Polyanskii \cite{JP20} obtained further refinements, culminating in the work of Jiang, Tidor, Yao, Zhang, and Zhao \cite{JTYZZ21} who determined $N^\R_{\alpha}(r)$ completely when $r$ is at least doubly exponential in $k / \alpha$, where $k = k\left(\frac{1-\alpha}{2\alpha}\right)$ is the corresponding spectral radius order (see the end of \Cref{intro_notation} for a definition). However, it is not possible to extend these recent results to the remaining regime $1/\alpha^2 - 2 \leq r \leq 2^{\Theta(1/\alpha^2)}$ because they all rely on the framework we introduced in \cite{BDKS18}, which uses Ramsey's theorem from graph theory. Note that the only other bounds known for general $\alpha$ in this regime were also obtained more recently by Yu \cite{Y17} and Glazyrin and Yu \cite{GY18} using semidefinite programming and properties of Gegenbauer polynomials. In particular, Yu showed that $N^\R_\alpha(r) \leq \binom{1/\alpha^2 - 1}{2}$ when $r \leq 3/\alpha^2 - 16$.

In this paper, we overcome the bottleneck of the Ramsey-theoretic approach by introducing a new method via orthogonal projection of matrices with respect to the Frobenius inner product, thereby obtaining a unified framework which allows us to significantly extend or improve almost all of the aforementioned results. We obtain several improved bounds in different regimes. For instance, one of our main results (\Cref{thm_bound1_r}) implies a bound which is simply the maximum of Yu's bound \cite{Y17} and the linear bound of \cite{BDKS18}, i.e.\ we show that
\[
N^\R_\alpha(r) \leq \max\left( \binom{1/\alpha^2 - 1}{2},  2r  \right) \text{ for all } \alpha \text{ and } r.
\] 
Moreover, we show that Yu's bound is tight if and only if the absolute bound is met in $1/\alpha^2 - 2$ dimensions and it is also tight up to a factor of two for infinitely many $\alpha$, while the bound of $2r$ is always tight up to a factor of two. We also determine $N^\R_{\alpha}(r)$ completely whenever $r$ is at least doubly exponential in $k \log(1/\alpha)$, where $k$ is the corresponding spectral radius order. Moreover, we also use our projection method to obtain a refinement of the first Welch bound \cite{W74} from coding theory.

Since the unit vectors which span a family of real equiangular lines have pairwise inner product $\alpha$ or $-\alpha$, this information can be encoded in a corresponding graph. Conversely, given any $d$-regular graph whose adjacency matrix has second largest eigenvalue $\lambda_2$ satisfying $d - \lambda_2 \leq (1-\varepsilon)n/2$ where $\varepsilon > 0$, there exists a family of equiangular lines which corresponds to it. If we assume that $\varepsilon$ is a constant, then roughly speaking, our main results for equiangular lines rely on the following novel lower bound 
\[
\lambda_2 \geq \Omega\left(d^{1/3} \right).
\] 
We derive this bound, as well as the related bound $\lambda_2 \geq \Omega\left( |\lambda_n|^{1/2}\right)$ (see \Cref{regular_bounds} for precise statements), by orthogonally projecting specially chosen matrices onto the span of the linear projection matrices $v_1 v_1^\intercal, \ldots, v_n v_n^\intercal$ corresponding to the given lines. Furthermore, these bounds are tight for strongly regular graphs corresponding to families of equiangular lines meeting the absolute bound in $\R^r$. Prior to this work, such bounds were only known for sufficiently sparse graphs (diameter $\geq 4$) via the Alon--Boppana theorem and so, our results may be seen as the first extension of this celebrated theorem to dense graphs. Furthermore, our results also rely on a bootstrapping argument using other new variants of the Alon--Boppana theorem (see \Cref{lem_beta_lower} and the proof of \Cref{lem_bootstrap}, as well as \Cref{lem_bootstrap_optimal}).

\begin{remark} \label{rem_rst}
After an earlier version of this paper was put on arXiv, Ihringer \cite{I23} obtained another proof of our inequality $\lambda_2 \geq \Omega\left( |\lambda_n|^{1/2}\right)$, interpreting it as a variant of the Krein bound for regular graphs. Furthermore, R\"aty, Sudakov, and Tomon \cite{RST23} (see also \cite{BRST23} for a short exposition) very recently gave alternative proofs of all of our new variants of the Alon--Boppana theorem for regular graphs. In particular, they noted that the bound $\lambda_2 \geq \Omega\left( d^{1/3} \right)$ follows from $\lambda_2 \geq \Omega\left( |\lambda_n|^{1/2}\right)$ and moreover, they pointed out the existence of certain infinite families of strongly regular graphs which imply that both bounds are asymptotically tight when $d = (1-o(1))(2n)^{3/4}$ or tight up to a multiplicative constant when $d = \Theta(n)$.
\end{remark}

We also consider the more general question of determining $N^\C_{\alpha}(r)$, the maximum number of complex equiangular lines in $\C^r$ with common Hermitian angle $\arccos(\alpha )$ (see \Cref{intro_complex}). Similarly to the real case, there is an absolute bound of $r^2$ and the relative bound holds for $r \leq 1/\alpha^2 - 1$, but no other bounds were known in the complementary regime $r \geq 1/\alpha^2 - 1$. Our approach generalizes to this setting, allowing us to obtain the first such bounds. Note that an important conjecture in quantum information theory due to Zauner (\Cref{conj_Zauner}) states that families of $r^2$ equiangular lines in $\C^r$ (SIC-POVMs/SICs) exist for all $d \in \N$ and since we are able to generalize the result of Yu to this setting, this provides a potential alternative route to establishing this conjecture, see concluding remark 4 in \Cref{concluding_remarks}.

\subsection{Structure of the paper}

This paper is organized as follows. In \Cref{intro_real}, \Cref{intro_complex}, and \Cref{intro_regular} we give more detailed introductions to and summaries of our main results on real equiangular lines, complex equiangular lines, and eigenvalues of regular graphs, respectively. \Cref{intro_notation} contains relevant notation and definitions that are used throughout the paper. In \Cref{section_projections}, we discuss how orthogonal projection of matrices can be used to derive the geometric inequalities which underly our results, how this approach relates to the Delsarte linear programming method, and as an additional application, obtain an improvement to the first Welch bound. Then in \Cref{section_r}, \Cref{section_c}, and \Cref{section_graphs} we derive our main results for real equiangular lines, complex equiangular lines, and regular graphs, respectively. Finally, in \Cref{concluding_remarks} we give some concluding remarks and directions for future research.

\subsection{Real equiangular lines} \label{intro_real}

Let $N^{\R}(r)$ be the maximum number of equiangular lines in $\R^r$. It is not hard to show that $N^{\R}(2) = 3$, with the optimal configuration being familiar to anyone who has cut a pizza pie into 6 equal slices using 3 cuts. In 1948, Haantjes \cite{H48} showed that $N^{\R}(3) = N^{\R}(4) = 6$, with an optimal configuration coming from the 6 diagonals of a regular icosahedron. The question of determining $N^\R(r)$ for an arbitrary $r$ was first formally posed in 1966 by van Lint and Seidel \cite{LS66}, and a few years later Gerzon (see \cite{LS73}) proved the \emph{absolute bound} 
\begin{equation*} \label{eq_absolute}
N^{\R}(r) \leq \binom{r+1}{2}.
\end{equation*}
This bound is known to be tight for $r = 2, 3, 7, 23$, with the constructions for $r = 7$ and $r = 23$ being based on the $E_8$ lattice in $\R^8$ and the Leech lattice in $\R^{24}$, see \cite{LS73} and references therein. Surprisingly, we do not know if there are any other $r$ for which this is the case. Moreover, the order of magnitude of $N^{\R}(r)$ was not even known until 2000, when de Caen \cite{C00} gave a construction showing that $N^{\R}(r) \geq \frac{2}{9}(r + 1)^2$ for infinitely many $r$.

In 1973, Lemmens and Seidel \cite{LS73} proposed to study the more refined quantity $N^{\R}_{\alpha}(r)$, the maximum number of equiangular lines in $\R^r$ with common angle $\arccos(\alpha)$, where $\alpha \in [0,1)$. They proved the \emph{relative bound}
\begin{equation*} \label{eq_relative}
    N^\R_{\alpha}(r) \leq \frac{1 - \alpha^2}{1 - \alpha^2 r} r \text{ for all } r < 1/\alpha^2, 
\end{equation*}
which is known to be tight if and only if the unit vectors $v_1, \ldots, v_n$ which span the lines form a \emph{tight frame}, i.e.\ $\sum_{i=1}^n{v_i v_i^\intercal = \frac{n}{r} I}$. Equiangular tight frames correspond to strongly regular graphs with certain parameters and subject to divisibility conditions, they are known to exist for many different values of $r$ and $n$, see the survey of Fickus and Mixon \cite{FM16}. Lemmens and Seidel \cite{LS73} also showed that a construction meeting the absolute bound must have $r = 1/\alpha^2 - 2$, so that the relative bound can be seen as a refinement of the absolute bound when $r \leq 1/\alpha^2 - 2$. They were particularly interested in the case where $1/\alpha$ is an odd integer due to a result of Neumann (see \cite{LS73}), who showed that if this is not the case, then $N^{\R}_{\alpha}(r) \leq 2r$. 

More recently, Bukh \cite{B16} showed that $N^\R_{\alpha}(r)$ is linear in $r$ when $\alpha$ is fixed, as well as conjecturing the asymptotic value of $N^{\R}_{\alpha}(r)$ as $r \rightarrow \infty$ whenever $1/\alpha$ is an odd integer. Together with Dr\"{a}xler, Keevash, and Sudakov \cite{BDKS18}, we significantly improved this result, showing that there exists a positive constant $C$ such that $N^\R_{\alpha}(r) \leq 2r - 2$ for all $r \geq 2^{C / \alpha^2}$, with equality if and only if $\alpha = 1/3$. Some of the ideas we proposed were further clarified and extended by Jiang and Polyanskii \cite{JP20}, who also generalized Bukh's conjecture to any $\alpha$. Using the framework developed in \cite{BDKS18}, together with a new bound on the maximum multiplicity of the second eigenvalue of a connected graph, Jiang, Tidor, Yao, Zhang, and Zhao \cite{JTYZZ21} were able to verify this conjecture in a strong form. To state their result, we first define the \emph{spectral radius order} $k(\lambda)$ to be the least number of vertices in a graph having spectral radius $\lambda$ (we set $k(\lambda) = \infty$ if there is no such graph). If the spectral radius order $k = k\left(\frac{1-\alpha}{2\alpha}\right) < \infty$, it was shown in \cite{JTYZZ21} that there exists a positive constant $C$ such that $N^\R_{\alpha}(r) = \left \lfloor \frac{k}{k-1} (r-1) \right \rfloor \text{ for all } r \geq 2^{2^{C k / \alpha}}$.

All of these recent results apply Ramsey's theorem in order to obtain a bound on the maximum degree of a corresponding graph which is exponential in $1/\alpha$. Therefore, such methods stop working when $r$ is much smaller than $2^{1/\alpha}$, in which case the only known bounds were obtained using semidefinite programming and Gegenbauer polynomials. Yu \cite{Y17} showed that $N^\R_{\alpha}(r) \leq \binom{1/\alpha^2 - 1}{2}$ for $1 / \alpha^2 - 2 \leq r \leq 3/\alpha^2  - 16$ and $\alpha \leq 1/3$, which can be seen as a refinement of the absolute bound when $r \geq 1/\alpha^2 - 2$ and later Glazyrin and Yu \cite{GY18} characterized the case of equality, as well as proving the universal bound $N^\R_{\alpha}(r) \leq O\left( r/\alpha^2 \right)$ for all $\alpha \leq 1/3$. Note that Yu's bound is tight if and only if the absolute bound is met in some dimension, but it is only relevant when $r$ is very close to $1/\alpha^2$. On the other hand, Glazyrin and Yu's universal bound works for all $r$, but it is off by a factor of $1/\alpha^2$ when $r$ is large relative to $1/\alpha$.

In this paper, we introduce a new method based on orthogonal projection of symmetric matrices with respect to the Frobenius inner product in order to unify and improve upon previous approaches to obtaining upper bounds on $N^\R_\alpha(r)$. As a result, we are able to extend our arguments from \cite{BDKS18} to the entire range $r \geq 1/\alpha^2 - 2$, thereby obtaining the following theorem which gives a bound that is the maximum of our linear bound from \cite{BDKS18} and the bound of Yu \cite{Y17}. Moreover, for the regime not covered by the relative bound, i.e.\ for all $r \geq 1/\alpha^2 - 2$, this theorem significantly extends the bound of Yu and is best possible either exactly or up to a small multiplicative constant.

\begin{theorem} \label{thm_bound1_r}
For all $0 < \alpha < 1$ and $r \in \N$, we have $N^\R_\alpha(r) \leq \max\left( \binom{1/\alpha^2 - 1}{2}, 2r  \right)$. Moreover, if $\alpha \leq 1/7$ and we define $r_\alpha = \frac{1}{2} \binom{1/\alpha^2 - 1}{2} + \frac{(1+\alpha)^2}{16 \alpha^2}$, then for all $r \leq r_\alpha$, we have
\[
N^\R_{\alpha}(r) \leq \binom{1/\alpha^2 - 1}{2}
\]
with equality only if the corresponding lines span a $(1/\alpha^2 - 2)$-dimensional subspace. Furthermore, for all $r > r_\alpha$, we have
\[
N^\R_{\alpha}(r) \leq 2r - \frac{(1+\alpha)^2}{8 \alpha^2}.
\]
\end{theorem}

\begin{remark} \label{rem_yu}
As previously mentioned, it is known that if the absolute bound is met in $\R^{r'}$, then the corresponding $\alpha$ must satisfy $r' = 1/\alpha^2 - 2$, in which case \Cref{thm_bound1_r} would imply that $N^\R_\alpha(r) =  \binom{1/\alpha^2 - 1}{2}$ for all $\frac{1}{\alpha^2} - 2 \leq r < r_\alpha = \frac{1 - O(\alpha)}{4 \alpha^4}$. Conversely, if $N^\R_\alpha(r) =  \binom{1/\alpha^2 - 1}{2}$ for some $\frac{1}{\alpha^2} - 2 \leq r < r_\alpha = \frac{1 - O(\alpha)}{4 \alpha^4}$, then  \Cref{thm_bound1_r} implies that $1/\alpha^2$ is an integer and the corresponding collection of lines meets the absolute bound in a $(1/\alpha^2 - 2)$-dimensional subspace. 
\end{remark}

The proof of \Cref{thm_bound1_r} follows via bounds on the largest eigenvalue of a corresponding Gram matrix, see the beginning of \Cref{section_r} for an outline. Moreover, in view of \Cref{rem_yu}, improving this theorem for general $\alpha$ when $\frac{1}{\alpha^2} - 2 \leq r \leq \frac{1}{4 \alpha^4} - \Theta\left(\frac{1}{\alpha^2} \right)$ would require showing that the absolute bound can be improved for all sufficiently large dimensions, which is a difficult open problem\footnote{The question of whether the absolute bound is tight in a given dimension is equivalent to the existence of a strongly regular graph with certain parameters, see Corollary 5.6 in \cite{W09}.}. On the other hand, when $r$ is much larger than $1/\alpha^4$, we obtain the following further improved bounds, in the same way Jiang and Polyanskii \cite{JP20} refined our approach in \cite{BDKS18}.

\begin{theorem} \label{thm_bound3_r}
Let $s \geq 2$ be an integer and let $0 < \alpha < 1$. If $r \gg 1/\alpha^{2s + 1}$, then 
\[
N^\R_{\alpha}(r) \leq \left(1 + \left( \frac{2\alpha}{1-\alpha} \right)^2 + \frac{1 + o(1)}{4 \cos^2 \left(\frac{\pi}{s+2}\right)} \right) r.
\]
In particular, if $r \geq 1/\alpha^{\omega(1)}$, then $N^\R_{\alpha}(r) \leq \left( \frac{5}{4} + \left( \frac{2\alpha}{1-\alpha} \right)^2 + o(1)\right) r$.
\end{theorem}
\noindent Finally, when $r$ is exponential in $1/\alpha$, we obtain the following extensions of the results of Jiang, Tidor, Yao, Zhang, and Zhao \cite{JTYZZ21}.

\begin{theorem} \label{thm_bound4_r}
There exists a constant $C > 0$, such that for all $0 < \alpha < 1$ with corresponding spectral radius order $k = k\left( \frac{1-\alpha}{2\alpha} \right)$, we have
\begin{enumerate}
    \item $N^\R_{\alpha}(r) = \left \lfloor \frac{k}{k-1} (r-1) \right \rfloor$ if $k < \infty$ and $r \geq 2^{1/\alpha^{C (k-1)}}$,
    \item $N^\R_{\alpha}(r) \leq r + \frac{C r \log(1/\alpha)}{\log{\log{r}}}$ if $k < \infty$ and $2^{1/\alpha^C} \leq r < 2^{1/\alpha^{C (k-1)}}$,
    \item $N^\R_{\alpha}(r) \leq r + \frac{C r \log(1/\alpha)}{\log{\log{r}}}$ if $k = \infty$ and $r \geq 2^{1/\alpha^C}$.
\end{enumerate}
\end{theorem}

\noindent Note that for any $\lambda$, the spectral radius order $k(\lambda) \geq \lambda + 1$. Moreover, in the special case that $\alpha = \frac{1}{2t - 1}$ for some integer $t \geq 2$, we have $k\left( \frac{1-\alpha}{2\alpha} \right) = t$ and thus we obtain the following corollary.
\begin{corollary}
There exists a constant $C > 0$ such that for all $r, k \in \N$ with $k \geq 2$ and $r \geq 2^{k^{C k}}$, \[
N^\R_{\frac{1}{2k-1}}(r) = \left \lfloor \frac{k}{k-1} (r-1) \right \rfloor.
\]
\end{corollary}

In order to help the reader interpret the significance of the new upper bounds in various regimes depending on how large $r$ is relative to $1/\alpha$, we provide a high level overview here, assuming that $\alpha \rightarrow 0$ and  that the spectral radius order $k$ satisfies $k = O(1/\alpha)$ (see the concluding remarks for a more detailed table summarizing these results). \Cref{thm_bound1_r} establishes an essentially best possible bound of $\binom{1/\alpha^2 -1}{2}$ when $1/\alpha^2 - 2 \leq r \leq O(1/\alpha^4)$, which was previously only known when $r = \Theta(1/\alpha^2)$. It also yields a linear bound of $2r$ when $r \geq \Omega(1/\alpha^4)$, which is tight up to a factor of 2 and improves on the previously best known bound of $O(r/\alpha^2)$ for $\Omega(1/\alpha^4) \leq r \leq 2^{O(1/\alpha^2)}$. Furthermore, \Cref{thm_bound3_r} yields even further improved bounds of the form $\gamma r$ where $\gamma < 2$, provided that $r$ is at least a polynomial in $1/\alpha$ with large degree. Item 2 of \Cref{thm_bound4_r} implies that the answer to the problem is $(1 + o(1))r$ when $r \geq 2^{1/\alpha^{\Omega(1)}} $, improving the previously best known bound of $2r$ when $2^{1/\alpha^{\Omega(1)}} \leq r \leq 2^{2^{O(1/\alpha^2)}}$. Finally, item 1 of \Cref{thm_bound4_r} extends the range for which the problem is completely solved from $r \geq 2^{2^{\Omega(1/\alpha^{2})}}$ to $r \geq 2^{1/\alpha^{\Omega(1/\alpha)}}$.

We also note that Lemmens and Seidel \cite{LS73} already determined that $N^{\R}_{1/3}(r) = \max\left(28, 2r - 2\right)$ for all $r \geq 7$ and moreover, they also conjectured that $N^{\R}_{1/5}(r) = \max\left(276, \lfloor 3(r-1)/2 \rfloor \right)$ for all $r \geq 23$, which was recently proven by Cao, Koolen, Lin, and Yu \cite{CKLY22}, building off of the work of Neumaier \cite{N89}. In view of these results, we propose the following more general conjecture.

\begin{conjecture} \label{conj_ambitious}
For all $0 < \alpha < 1$ with corresponding spectral radius order $k = k\left( \frac{1-\alpha}{2\alpha} \right) < \infty$,
\[
N^\R_\alpha(r) \leq \max\left( \binom{1/\alpha^2 - 1}{2}, \left \lfloor \frac{k}{k-1}(r-1) \right \rfloor \right).
\]
\end{conjecture}
\noindent Even if this ambitious conjecture is false in general, it would already be very interesting to see if it holds in the case where $1/\alpha$ is an odd integer. Note that \Cref{thm_bound1_r} establishes a weak version of \Cref{conj_ambitious} and moreover, if one could significantly improve Jiang, Tidor, Yao, Zhang, and Zhao's bound \cite{JTYZZ21} on the multiplicity of the second eigenvalue of a connected graph, then \Cref{conj_ambitious} would follow from our arguments.

 
We conclude this subsection with some lower bounds for $N^\R_{\alpha}(r)$, thereby showing that there exists a constant $C$ such that \Cref{thm_bound1_r} is best possible up to a small multiplicative constant when $r \geq C/\alpha^2$. For all $\alpha \in [0,1)$ and $r \in \N$, it is not hard to see that $N^\R_\alpha(r) \geq r$. Indeed, if we rotate each standard basis in $\R^r$ by the same angle towards the all ones vector $\one$, we can obtain a collection\footnote{More precisely, for $i \in [r]$, let $v_i = \frac{1}{r}\left(\sqrt{1-\alpha + \alpha r} - \sqrt{1-\alpha} \right) \mathbbm{1} + \left( \sqrt{1-\alpha} \right) e_i$ in $\R^r$. Then $\{v_1,\ldots, v_r\}$ is a collection of unit vectors in $\R^r$ with pairwise inner product $\alpha$. }  of $r$ unit vectors with all pairwise inner products $\alpha$. Using this fact, we obtain the following immediate corollary of \Cref{thm_bound1_r}.

\begin{corollary} \label{cor_linear_growth}
For all $0 < \alpha < 1$ and $r \geq \Omega(1/\alpha^4)$, we have $N^\R_\alpha(r) = \Theta(r)$.
\end{corollary}

In view of the aforementioned result of Neumann, we can only hope to obtain lower bounds which are superlinear in $r$ when $1/\alpha$ is an odd integer. Moreover, it is known that all constructions of $\Omega(r^2)$ lines in $\R^r$ satisfy $r = \Theta(1/\alpha^2)$. In particular, de Caen's construction \cite{C00} works for all dimensions $r_t = 3 \cdot 2^{2t-1} - 1$ where $t \in \N$. Using the fact that $N^{\R}_{\alpha}(r)$ is a non-decreasing function of $r$, we immediately obtain the following corollary of de Caen's construction and \Cref{thm_bound1_r}.

\begin{corollary} \label{cor_lower_super}
There exists a constant $C > 0$ and an infinite sequence $\{\alpha_t\}_{t \in \N}$ with $\alpha_t \rightarrow 0$ such that for all $\alpha = \alpha_t$ and $\frac{C}{\alpha^2} \leq r \leq O\left( \frac{1}{\alpha^4} \right)$, we have $N^\R_\alpha(r) = \Theta(1/\alpha^4)$.
\end{corollary}

Using \Cref{thm_bound1_r}, together with a family of strongly regular graphs constructed by Metz (see \Cref{rem_Metz}), we also note that there exists a constant $C$ such that for any $\alpha = \frac{1}{2q-1}$ where $q$ is a prime power, $N^\R_\alpha(r) = \Theta\left(1/\alpha^4\right)$ when $C/\alpha^3 \leq r \leq O\left( 1/\alpha^4 \right)$.


\subsection{Complex equiangular lines} \label{intro_complex}

For a pair of 1-dimensional subspaces $U, V \subseteq \C^r$, i.e.\ complex lines through the origin, the quantity $\arccos{|\inprod{u}{v}|}$ is the same for any choice of unit vectors $u \in U, v \in V$, and so we may define this quantity to be the \emph{Hermitian angle} between $U$ and $V$, see e.g.\ \cite{S01}. Given $n$ complex lines $l_1, \ldots, l_n$ passing through the origin in $\C^r$, we say they are \emph{equiangular} if there exists $\theta \in (0, \pi/2]$ such that the Hermitian angle between $l_i$ and $l_j$ is $\theta$ for all $i \neq j$, in which case we call $\theta$ the \emph{common Hermitian angle}. As in the real case, we define $N^\C_{\alpha}(r)$ to be the maximum number of complex equiangular lines in $\C^r$ with common Hermitian angle $\arccos(\alpha)$ and $N^\C(r) = \max_{\alpha \in [0,1)}{N^\C_{\alpha}(r)}$.

The earliest results on complex equiangular lines go back to the 1975 work of Delsarte, Goethals, and Seidel \cite{DGS75}. Analogous to the real case, they proved the absolute bound 
\[
N^\C(r) \leq r^2
\]
and gave matching lower bound constructions in $\C^2$ and $\C^3$. Zauner \cite{Z11} was the first to make the connection between complex equiangular lines and quantum theory, as well as showing that $N^\C(r) = r^2$ for $r \leq 5$. This lead him to conjecture that the same holds for all $r \in \N$, and moreover that extremal constructions can be obtained as the orbit of some vector under the action of a Weyl--Heisenberg group.
\begin{conjecture}[Zauner] \label{conj_Zauner}
$N^\C(r) = r^2$ for all $r \in \N$.
\end{conjecture}

Since the work of Renes, Blume-Kohout, Scott, and Caves \cite{RBSC04}, collections of $r^2$ complex equiangular lines in $\C^r$ have come to be known in quantum information theory as symmetric, informationally complete, positive operator-valued measures, or \emph{SIC-POVM}s/\emph{SIC}s for short. Unlike the real case, Zauner's conjecture has turned out to be true for more than 100 different values of the dimension $r$, including all $r \leq 40$ and as large as $r = 19603$, see \cite{GS17, S21, ABGHM22} and the survey paper of Fuchs, Hoang, and Stacey \cite{FHS17} for more information. In particular, Appleby, Bengtsson, Grassl, Harrison, and McConnell \cite{ABGHM22} provide a recipe which they conjecture produces a SIC in $\C^r$ for any $r \in \N$ such that $r - 3$ is a perfect square. We also note that Sz\"oll\H{o}si \cite{S14} classified all SICs in $\C^3$.
 
SICs turn out to be quite remarkable objects, having applications in quantum state tomography \cite{CFS02} and quantum cryptography \cite{FS03}, as well as being candidates for a ``standard quantum measurement'' in the foundations of quantum mechanics, most notably in QBism \cite{FS19}. Moreover, they also have applications in high-precision radar and speech recognition and it has been suggested that a SIC in $\C^{2048}$ is worth patenting (see the last paragraph of \cite{FHS17}). Finally, it is very intriguing to note that Zauner's conjecture is related to algebraic number theory with a connection to Hilbert's twelfth problem \cite{B17, AFMY17, AFMY20, ABGHM22}.  

Despite all of the research on $N^\C(r)$, very little is known about $N^\C_{\alpha}(r)$. Delsarte, Goethels, and Seidel \cite{DGS75} proved the analogous relative bound 
\[
N^\C_{\alpha}(r) \leq \frac{1 - \alpha^2}{1 - \alpha^2 r} r \text{ for all } r < 1 / \alpha^2,
\]
which is tight if and only if the unit vectors $v_1, \ldots, v_n$ spanning the lines form a tight frame in $\C$, i.e.\ $\sum_{i=1}^n{v_i v_i^*}= \frac{n}{r} I$. To the best of our knowledge, no bound for general $\alpha$ is known for the complementary regime $r \geq 1 / \alpha^2$. Since our projection method generalizes to the complex setting, we are able to obtain such bounds. Our results depends on a new geometric inequality for complex equiangular lines which is tight for any SIC and may therefore be of independent interest in quantum theory, see \Cref{rem_SIC}.

The following theorem is a weaker, but more general version of \Cref{thm_bound1_r}, which generalizes the bound of Yu \cite{Y17} and implies that $N^\C_\alpha(r)$ is linear in $r$ when $\alpha$ is fixed. Just like \Cref{thm_bound1_r}, the proof of \Cref{thm_bound1_c} follows from bounds on the largest eigenvalue of a corresponding Gram matrix. 

\begin{theorem} \label{thm_bound1_c}
Let $0 < \alpha < 1$ and define $r^\C_\alpha =  \frac{(1-\alpha)(1-\alpha^2)}{\alpha^3}  + 3\frac{1+\alpha}{\alpha^2}$. For all $r \leq r^\C_\alpha$, we have 
\[
N^\C_{\alpha}(r) \leq \left (  \frac{1}{\alpha^2} - 1\right)^2
\] 
with equality if and only if the corresponding lines form a SIC in $1/\alpha^2 - 1$ dimensions. Otherwise if $r > r^\C_\alpha$, then
\[
N^\C_{\alpha}(r) \leq \frac{1+\alpha}{\alpha} r   + 3\frac{1+\alpha}{\alpha^3}.
\]
\end{theorem}

\begin{remark} \label{ref_yu_c}
Similarly to the real case, it is well known (see e.g.\ \cite{FM16}) that if there exists a SIC in $\C^r$, then $r = 1/\alpha^2 - 1$, in which case \Cref{thm_bound1_c} implies that for all $\frac{1}{\alpha^2} - 1 \leq r' \leq  r^\C_\alpha = \frac{1-O(\alpha)}{\alpha^3}$, we have $N^\C_\alpha(r') = (1/\alpha^2 - 1)^2$. Conversely, if $N^\C_\alpha(r') = (1/\alpha^2 - 1)^2$ for some $\frac{1}{\alpha^2} - 1 \leq r' \leq  r^\C_\alpha$, then  \Cref{thm_bound1_c} implies that $1/\alpha^2$ is an integer and the corresponding collection of lines forms a SIC in a $(1/\alpha^2 - 1)$-dimensional subspace.
\end{remark}

In view of the fact that the existence of a SIC in $\C^r$ has been verified for many different $r$, we note that \Cref{ref_yu_c} is more important than \Cref{rem_yu}. For instance, it shows that Zauner's conjecture implies the following stronger conjecture.

\begin{conjecture} \label{conj_c}
There exists a constant $C$ such that for all $0 < \alpha < 1$ satisfying $\frac{1}{\alpha^2} \in \N$ and for all $\frac{1}{\alpha^2} - 1 \leq r <  \frac{1 - C \alpha}{\alpha^3}$, we have 
\[
N^\C_\alpha(r) = \left( \frac{1}{\alpha^2} - 1\right)^2.
\]
\end{conjecture}

We also note that Godsil and Roy \cite{GR09} gave a construction of $r^2 - r + 1$ many equiangular lines in $\C^r$ with cosine of the common angle $\alpha = \sqrt{\frac{r-1}{r^2}}$ whenever $r-1$ is a prime power. This allows us to immediately obtain the following asymptotic variant of \Cref{conj_c} as a corollary of \Cref{thm_bound1_c}.

\begin{corollary} \label{cor_godsil_roy}
Let $0 < \alpha < 1$ be such that $\frac{1 + \sqrt{1 - 4\alpha^2}}{2 \alpha^2} - 1$ is a prime power and assume $\alpha \rightarrow 0$. Then for all $\frac{1 + \sqrt{1 - 4\alpha^2}}{2 \alpha^2}  \leq r <  \frac{1 - o(1)}{\alpha^3}$, we have 
\[
N^\C_\alpha(r) = \frac{1-o(1)}{\alpha^4}.
\]
\end{corollary}

\subsection{Eigenvalues of regular graphs} \label{intro_regular}

Let $G$ be a $k$-regular graph whose adjacency matrix $A $ has second eigenvalue $\lambda_2$ and last eigenvalue $\lambda_n$. There is by now a lot of literature on how the size of $\lambda_2$ (and $\lambda_n$) relate to other properties of $G$, see \cite{BH11, CS95, CDS98} for more information. Inspired by the fact that $G$ is connected if and only if the spectral gap $k - \lambda_2 > 0$, Fiedler \cite{F73} studied the spectral gap as an algebraic measure of the connectivity of a graph and showed how it relates to the usual combinatorial notions of vertex and edge connectivity. It is also known that the spectral gap and $\max( \lambda_2, -\lambda_n)$ are related to the mixing time of random walks on a graph and that this generalises to other Markov chains, see e.g.\ \cite{DR93, DS91}. Moreover, Alon and Milman \cite{A86, AM85} showed that the spectral gap as well as $\max(\lambda_2, -\lambda_n)$ are closely related to combinatorial notions of graph expansion, which can also be interpreted as discrete isoperimetric inequalities. In particular, Chung, Graham, and Wilson \cite{CGW89} showed that $\max(\lambda_2, -\lambda_n)$ determines how pseudo-random a graph is. We remark that the theory of graphs having good expansion properties is not only deep and interesting on its own, but also has extensive applications in mathematics and computer science \cite{HLW06, KS06}, as well as applications to other sciences, see e.g.\ \cite{CDS98, T107}.

Given the above, it is interesting to determine how large the spectral gap $k - \lambda_2$ can be, or equivalently, how small $\lambda_2$ can be as a function of $k$. If $G$ has diameter $D$, the well-known Alon--Boppana theorem \cite{A86, N91} gives the lower bound $\lambda_2 \geq 2\sqrt{k-1} - (2 \sqrt{k-1} - 1)/\lfloor D/2 \rfloor$. As $D \rightarrow \infty$, this bound approaches $2\sqrt{k-1}$, which is known to be tight due to the existence of Ramanujan graphs, see e.g.\ \cite{HLW06}. Moreover, Alon conjectured and Friedman \cite{F08} proved that for any $\varepsilon > 0$, ``most'' $k$-regular graphs on sufficiently many vertices will satisfy $\lambda_2 \leq 2\sqrt{k-1} + \varepsilon$. However, the Alon--Boppana theorem doesn't say much when the diameter $D \in \{2, 3\}$ and even though this theorem has been improved and generalized in various ways (see e.g.\ \cite{F93, H05, J19}), it seems that all previously known results require the assumption that the graph is sufficiently sparse. 

In this paper, we make use of the connection between graphs and equiangular lines in order to obtain the first generalization of the Alon--Boppana theorem which holds for dense $k$-regular graphs with no assumption on the diameter. We only require the following assumption on the spectral gap of $G$,
\begin{equation} \label{spectral_gap}
    k - \lambda_2 < \frac{n}{2}.
\end{equation}
Indeed, such an assumption is necessary if we want to have a nontrivial lower bound on $\lambda_2$, since the complete bipartite graph with parts of size $n/2$ has $\lambda_2 = 0$. Moreover, note that any $k$-regular graph $G$ with $k < n/2$ must have $\lambda_2 \geq 0$ (see \Cref{lem_spectral_history}), and so it must satisfy \eqref{spectral_gap}. Additionally, the complement of any $k$-regular graph is $(n-1-k)$-regular, so if $k \geq n/2$ then \eqref{spectral_gap} holds for the complement graph. Finally, the bounds we obtain are tight for strongly regular graphs corresponding to families of equiangular lines meeting the absolute bound. To state the result, we let $\m_G(\lambda_2)$ denote the multiplicity of $\lambda_2$ in $A$.

\begin{theorem} \label{regular_bounds}
Let $k, n \in \N$ and let $G$ be a $k$-regular graph with second eigenvalue $\lambda_2 = \lambda_2(G)$ and last eigenvalue $\lambda_n = \lambda_n(G)$. If the spectral gap satisfies $k - \lambda_2 < n/2$, then we have
\[
2\left( k - \frac{(k-\lambda_2)^2}{n} \right) \leq \frac{ \lambda_2 (\lambda_2 + 1) (2 \lambda_2 + 1)}{1 - \frac{2(k-\lambda_2)}{n}} - \lambda_2(3\lambda_2 + 1),
\]
and
\[
-\lambda_n \leq \frac{ \lambda_2(\lambda_2 + 1) }{1 - \frac{2(k-\lambda_2)}{n} } - \lambda_2,
\]
with equality in both whenever $n = \binom{n - \m_G(\lambda_2) + 1}{2} - 1$, i.e.\ when $G$ is a strongly regular graph corresponding to a family of $\binom{r+1}{2}$ equiangular lines in $\R^r$ for $r = n - \m_G(\lambda_2)$.
\end{theorem}
\begin{remark} \label{rem_regular}
$k - \lambda_2 < n/2$ implies that $k + \lambda_2 < 2 \left( k - \frac{(k-\lambda_2)^2}{n } \right)$, so that the first inequality of \Cref{regular_bounds} gives an upper bound on the degree $k$. 
\end{remark}

Note that families of equiangular lines meeting the absolute bound $\binom{r+1}{2}$ in $\R^r$ are known to exist for $r = 2,3, 7, 23$, and the corresponding graphs are the edge, the pentagon, the complement of the Schl\"afli graph, and the McLaughlin graph, so that the bounds of \Cref{regular_bounds} are tight for such graphs. If we further assume that the spectral gap is bounded away from $n/2$, we immediately obtain the following corollary.

\begin{corollary} \label{cor_alon_boppana}
Let $k, n \in \N$ and let $G$ be a $k$-regular graph with second eigenvalue $\lambda_2 = \lambda_2(G)$ and last eigenvalue $\lambda_n = \lambda_n(G)$. If the spectral gap satisfies $k - \lambda_2 \leq (1-\varepsilon) \frac{n}{2}$ for some constant $\varepsilon > 0$, then we have
\[
\lambda_2 \geq \Omega\left( k^{1/3} \right)  \qquad \text{and} \qquad \lambda_2 \geq \Omega\left(  \sqrt{-\lambda_n} \right).
\]
In particular, if $G$ is bipartite then $\lambda_2 \geq \Omega(\sqrt{k}).$ Moreover, if the spectral gap $k - \lambda_2 = o(n)$ then
\[
\lambda_2 \geq (1-o(1)) k^{1/3}  \qquad \text{and} \qquad \lambda_2 \geq (1-o(1))  \sqrt{-\lambda_n},
\]
so that if $G$ is bipartite, we have $\lambda_2 \geq (1 - o(1)) \sqrt{k}$.
\end{corollary}


As mentioned in \Cref{rem_rst}, R\"aty, Sudakov, and Tomon \cite{RST23} noted that R. Metz (see \cite{BL84}, Section 7.A) constructed an infinite family of strongly regular graphs satisfying $k = (1-o(1))(2n)^{3/4}$, $\lambda_2 = (1-o(1)) k^{1/3}$, and $-\lambda_n = \lambda_2^2$ and moreover, that for a fixed odd prime $p$, Davis, Huczynska, Johnson, and Polhill \cite{DHJP24} recently constructed an infinite family of strongly regular graphs satisfying $k = (1-o(1))n/p$, $\lambda_2 = (1-o(1)) (p k)^{1/3}$, and $-\lambda_n = (1-o(1)) \lambda_2^2 / p$. It follows that \Cref{cor_alon_boppana} is tight when $k = (1-o(1))(2n)^{3/4}$ or $k = \Theta(n)$. 

\begin{remark} \label{rem_Metz}
Note that for any prime power $q$, if we let $\alpha = 1/(2q - 1)$, then the aformentioned construction of Metz leads to a collection of $\Theta(1/ \alpha^4)$ equiangular lines with common angle $\arccos(\alpha)$ in $\R^r$ with $r = \Theta(1/\alpha^3)$, matching the upper bound of \Cref{thm_bound1_r}.
\end{remark}


\subsection{Notation and definitions} \label{intro_notation}

Let $\N, \R, \C$ denote the set of positive integers, field of real numbers, and field of complex numbers, respectively. For all $n \in \N$, we define $[n] = \{1, \ldots, n\}$. We let $S^1(\alpha) = \{z \in \C : |z| = \alpha \}$ denote the circle of radius $\alpha$ centered at $0$ in the complex plane. For any quantities $f,g$ we write $f \leq O(g)$ or $g \geq \Omega(f)$, if $f,g$ are implicitly assumed to be functions of an increasing parameter $t \in \N$ and there exist $C > 0, t_0 \in \N$ such that $f(t) \leq C g(t)$ for all $t \geq t_0$. Moreover, we write $f = \Theta(g)$ if $f \leq O(g)$ and $f \geq \Omega(g)$. Also, we write $f = o(g)$, $g = \omega(f)$ or $f \ll g$ whenever $f(t)/g(t) \rightarrow 0$ as $t \rightarrow \infty$. Finally, when doing calculations, we will often make use of the inequality $\sqrt{1+x} \leq 1 + x/2$ and its natural consequence $\sqrt{x^2 + y} \leq x + \frac{y}{2x}$.

\subsubsection*{Linear algebra}

Let $\F$ be a field and let $S$ be a finite set. We let $\F^{S}$ denote the vector space of functions from $S$ to $\F$, we let $\one$ denote the \emph{all ones vector} (i.e.\ the constant function with value $1$), and for $i \in S$, we let $e_i$ denote the \emph{standard basis vector} in $\F^{S}$ given by $(e_i)(j) = 1$ if $i = j$ and $(e_i)(j) = 0$ otherwise. More generally, given another finite set $T$, we shall consider functions in $\F^{S \times T}$ as matrices indexed by $S \times T$ and we also do no distinguish between matrices in $\F^{S \times T}$ and the corresponding linear maps from $\F^T$ to $\F^S$. 
For a matrix $M \in \F^{S \times S}$, we let $\rk(M)$ and $\tr(M)$ denote its rank and trace, respectively and for a pair of matrices $M, N \in \F^{S \times S}$, we let $M \odot N$ denote their \emph{Hadamard product}, given by $(M \odot N)(i,j) = M(i,j) N(i,j)$. 

Given $M \in \C^{S \times T}$, we let $M^{\intercal}$ and $M^*$ denote the \emph{transpose matrix} and \emph{adjoint matrix} in $\C^{T \times S}$. We also let $\Se_S = \{M \in \R^{S \times S} : M^{\intercal} = M\}$ denote the space of all real symmetric matrices indexed by $S \times S$ and let $\inprod{u}{v} = u^* v$ denote the \emph{standard inner product} of $u, v \in \C^S$. For any Hermitian matrix $M \in \C^{S \times S}$, we let $\lambda_i = \lambda_i(M) \in \R$ for $i \in S$ denote its eigenvalues ordered so that $\lambda_1 \geq \ldots \geq \lambda_n$ and we define $\m_M(\lambda) = |\{j: \lambda_j = \lambda\}|$ to be the \emph{multiplicity} of $\lambda$.  Finally, letting $u_1, \ldots, u_n$ be an orthonormal basis of corresponding eigenvectors such that $M = \sum_{i=1}^{n}{\lambda_i u_i u_i^*}$, we define the \emph{Moore--Penrose inverse} of $M$ to be $M^{\dag} = \sum_{i : \lambda_i \neq 0}{ \frac{1}{\lambda_i} u_i u_i^* }$. 

For convenience, we let $\F^n = \F^{[n]}$, $\F^{n \times S} = \F^{[n] \times S}$, and $\Se_n = \Se_{[n]}$.

\subsubsection*{Inner product spaces}

Let $\F \in \{\R, \C\}$ and let $V$ be a vector space over $\F$. A map $\inprod{\cdot}{\cdot}_{\text{P}} \colon V \times V \rightarrow \F$ is called an \emph{inner product} if it satisfies
\begin{enumerate} 
    \item $\inprod{u}{\cdot}_{\text{P}} \colon V \rightarrow \F$ is linear for all $u \in V$,
    \item $\overline{\inprod{u}{v}_{\text{P}} } = \inprod{v}{u}_{\text{P}}$ for all $u,v \in V$,
    \item $\inprod{u}{u}_{\text{P}} \ = 0$ implies $u = 0$ for all $u \in V$,
\end{enumerate}
and we say that $V$ is an \emph{inner product space} with respect to $\inprod{\cdot}{\cdot}_{\text{P}}$, letting $||\cdot||_{\text{P}} \colon V \rightarrow \F$ denote the corresponding \emph{norm}. For any $u \in V$, we let $u^{\#} \colon V \rightarrow \F$ denote the \emph{adjoint} of $u$ defined by $u^{\#}v = \inprod{u}{v}_{\text{P}}$ for all $v \in V$. More generally, if $U$ and $V$ are finite-dimensional vector spaces with inner products $\inprod{\cdot}{\cdot}_{\text{P}}$ and $\inprod{\cdot}{\cdot}_{\text{Q}}$, respectively and $L \colon U \rightarrow V$ is a linear map, then we let $L^{\#} \colon V \rightarrow U$ denote the \emph{adjoint map}\footnote{If we choose bases for $U$ and $V$, then we can represent the linear map $L$ by a corresponding matrix $[L]$ with respect to those bases, in which case the adjoint map $L^{\#}$ is represented by the corresponding adjoint matrix $[L]^*$.}, which satisfies $\inprod{L^{\#}v}{u}_{\text{P}} = \inprod{v}{Lu}_{\text{Q}}$ for all $u \in U, v \in V$.

Let $V$ be an inner product space with respect to $\inprod{\cdot}{\cdot}_{\text{P}}$ and let $\Ce \subseteq V$ be a finite set of vectors. We define $M_{\Ce} \in \F^{\Ce \times \Ce}$ to be their \emph{Gram matrix}, defined by $(M_{\Ce})(u,v) = \inprod{u}{v}_{\text{P}}$. Moreover, for any subset $S \subseteq \F$, we say that a collection of unit vectors $\Ce$ is a \emph{spherical $S$-code} if $\inprod{u}{v}_{\text{P}} \in S$ for all $u \neq v \in \Ce$. Finally, given a finite set $S$, we let $\inprod{\cdot}{\cdot}_F \colon \F^{S \times S} \times \F^{S \times S} \rightarrow \F$ denote the \emph{Frobenius inner product}, which is defined by $\inprod{A}{B}_F = \tr\left(A^* B \right)$ for all $A,B \in \F^{S \times S}$.

\subsubsection*{Graph theory}
A \emph{graph} $G$ consists of a finite set $\V(G)$ whose elements are called \emph{vertices} and a collection of unordered pairs of vertices $\E(G) \subseteq \{\{u,v\} : u, v \in \V(G), u \neq v \}$, which are called \emph{edges}. For convenience, whenever $u$ and $v$ are vertices of a graph $G$, we write $uv$ instead of $\{u, v\}$ in order to represent the unordered pair. We let $G^c$ denote the complement of $G$ and we write $H \subseteq G$ when $H$ is a subgraph of $G$. For any vertex $v \in V(G)$, we let $N_G(v) = \{u \in V(G) : uv \in E(G) \}$ denote its \emph{neighborhood} and let $d_G(v) = |N_G(v)|$ denote its \emph{degree}. We also define $\overline{d_G} = \frac{1}{|V(G)|}\sum_{v \in V(G)}{d_G(v)}$ to be the \emph{average degree} of $G$ and $\Delta_G = \max_{v \in V(G)}{d_G(v)}$ to be the \emph{maximum degree} of $G$. A graph $G$ is called $d$-\emph{regular} if $d_G(v) = d$ for all $v \in V(G)$. We let $K_n$ denote the complete graph on $n$ vertices and for $r \geq 2$, let $K_{n_1, \ldots, n_r}$ denote the complete multipartite graph with $r$ parts having sizes $n_1, \ldots, n_r$. 

For any graph $G$, we let $A_G \in \R^{V(G) \times V(G)}$ denote the corresponding \emph{adjacency matrix}, defined by $A_G(u,v) = 1$ if $uv \in E(G)$ and $A_G(u,v) = 0$ otherwise. Since $A_G$ is symmetric, it has a full set of real eigenvalues $\lambda_1(G) \geq \ldots \geq \lambda_n(G)$, and we equivalently say that $\lambda_i = \lambda_i(G)$ is the $i$th eigenvalue of $G$ with multiplicity $\m_G(\lambda_i) = \m_{A_G}(\lambda_i)$. In particular, $\lambda_1$ is called the \emph{spectral radius} of $G$ and we say that $\lambda_1 - \lambda_2$ is the \emph{spectral gap} of $G$. Finally, the \emph{spectral radius order} $k(\lambda)$ is defined to be the least number of vertices in a graph having spectral radius $\lambda$, where we set $k(\lambda) = \infty$ if no such graph exists.

\section{Orthogonal projection of matrices} \label{section_projections}

In this section, we discuss how orthogonal projection of matrices with respect to the Frobenius inner product can be used to obtain new inequalities for positive semidefinite matrices (equivalently Gram matrices of vectors) which underlie our main results. We also apply this method to obtain a strengthening of the first Welch bound in coding theory, in particular giving another proof of the relative bound. Finally, we discuss how our approach relates to the Delsarte linear programming method.

\subsection*{Real equiangular lines}

The usual way one studies real equiangular lines in $\R^r$ is by choosing a unit vector along each line, thereby obtaining a collection of vectors $\Ce$ with all pairwise inner products lying in the set $\{\alpha, -\alpha\}$ where $\alpha$ is the cosine of the common angle, i.e.\ $\Ce$ is a spherical $\{\alpha, -\alpha \}$-code in $\R^r$. In our previous work together with Dr\"{a}xler, Keevash, and Sudakov \cite{BDKS18}, the arguments rested upon finding a sufficiently large subset of $\Ce$ with the property that all pairwise inner products are $\alpha$, i.e.\ a large regular simplex, and then projecting remaining vectors onto the orthogonal complement of the span of this simplex. However, instead of looking for such a simplex in $\Ce$, in this paper we use the fact that the projection matrices $\{ v  v^{\intercal} : v \in \Ce \}$ already form a large regular simplex in $\Se_r$, the space of $r \times r$ real symmetric matrices with respect to the Frobenius inner product. This result is well known and indeed, when combined with the fact that the dimension of $\Se_r$ is $\binom{r+1}{2}$, it yields a proof of the absolute bound. 

Our main contribution is the observation that one can obtain new geometric inequalities by orthogonally projecting a matrix onto the span of $\{ v  v^{\intercal} : v \in \Ce \}$ and using the fact that the Frobenius norm can only decrease. A benefit of this approach is that if the absolute bound is met, i.e.\ if $n = \binom{r+1}{2}$, then $\{ v  v^{\intercal} : v \in \Ce \}$ span the whole space $\Se_r$, and so our inequalities become tight. 

So what kind of matrices should we project onto $\{ v  v^{\intercal} : v \in \Ce \}$? Since an orthogonal projection of an asymmetric matrix onto a space of symmetric matrices necessarily reduces its Frobenius norm and since we are interested in obtaining inequalities that have the possibility of being tight, we restrict ourselves to projecting symmetric matrices. Perhaps the simplest matrix to try is the identity and indeed, this leads to a proof of the relative bound (see \Cref{rem_relative}). It is also natural to try to construct the matrix as a tensor product of some specially chosen vectors. In particular, in this paper we observe that projecting $\left( \sum_{v \in \Ce}{x(v) v} \right) u^\intercal + u \left(  \sum_{v \in \Ce}{x(v) v} \right)^\intercal$, where $u \in \Ce$ and $x$ is a unit eigenvector corresponding to eigenvalue $\lambda$ of the Gram matrix $M_\Ce$, leads to the following inequality involving $\alpha, n, x(u)$, and $\lambda$, which is tight if $n = \binom{r+1}{2}$. 


\begin{theorem} \label{main_r}
Let $0 < \alpha < 1$ and let $\Ce$ be a spherical $\{ \alpha, -\alpha\}$-code in $\R^r$ with Gram matrix $M = M_{\Ce}$ and $n = |\Ce|$.  If $\lambda$ is an eigenvalue of $M$ with corresponding unit eigenvector $x$, then
\[
\lambda \left(\frac{\lambda^2}{\alpha^2 n + 1 - \alpha^2} - \frac{1-\alpha^2}{2\alpha^2} \right) x(u)^2 \geq \lambda - \frac{1-\alpha^2}{2\alpha^2}
\]
for all $u \in \Ce$, with equality if $n = \binom{r+1}{2}$.
\end{theorem}

For the proof of \Cref{main_r}, recall that given a linear map $\Le \colon \R^{\Ce} \rightarrow \Se_r$ where $\Se_r$ is equipped with the Frobenius inner product and $\R^\Ce$ is equipped with the standard inner product, the adjoint $\Le^{\#} \colon \Se_r \rightarrow \R^n$ satisfies $\inprod{\Le^{\#}M}{v} = \inprod{M}{\Le v}_F$ and for a single matrix $M \in \Se_r$, the adjoint $M^{\#} \colon \Se_r \rightarrow \R$ satisfies $M^{\#}N = \inprod{M}{N}_F$. We will also make use of the well-known fact that $\inprod{x y^\intercal}{z w^\intercal}_F = \inprod{x}{z} \inprod{y}{w}$ for any $x,y,z,w \in \R^r$.

\begin{proof}[Proof of \Cref{main_r}]
We define $\W \colon \R^\Ce \rightarrow  \Se_r$ to be the linear map given by $\W e_v = v v^\intercal$ for all $v \in \Ce$ and consider the adjoint map $\W^\# \colon \Se_r \rightarrow \R^\Ce$ with respect to the Frobenius inner product on $\Se_r$ and the standard inner product on $\R^\Ce$. Note that if we let $\Ce = \{v_1, \ldots, v_n\}$, then since any $r \times r$ matrix $M$ can be viewed as a vector $\vect(M)$ with $r^2$ coordinates, one can equivalently think of $\W$ in the following proof as an $r^2 \times n$ matrix $W$ with $i$th column given by $\vect(v_i v_i^\intercal)$. In this case, the Frobenius inner product $\inprod{M}{N}_F$ of two $r \times r$ matrices $M, N$ simply becomes the standard inner product $\inprod{\vect(M)}{\vect(N)}$ in $\R^{r^2}$ and the adjoint map $\W^{\#}$ becomes the $n \times r^2$ transpose matrix $W^\intercal$.  Also, $V$ is then just a matrix with $i$th column given by $v_i$. As such, the reader may, if they so choose, read the following argument using these substitutions with no loss of generality.

Observe that $\W^\# \W \colon \R^\Ce \rightarrow \R^\Ce$ is a matrix satisfying
\[ 
(\W^\# \W)(u,v)
= \inprod{u u^\intercal}{v v^\intercal}_F 
= \inprod{u}{v}^2 
= 
\begin{cases}
\alpha^2 &\text{ if } u \neq v \\
1 &\text{ if } u = v
\end{cases}
\]
for all $u, v \in \Ce$, i.e.\ we have $\W^{\#} \W = (1-\alpha^2) I + \alpha^2 J$ which is invertible with $(\W^{\#} \W)^{-1} = \frac{1}{1-\alpha^2} \left( I - \frac{\alpha^2}{\alpha^2 n + 1 - \alpha^2} J \right)$. It follows that $\Pe = \W (\W^{\#} \W)^{-1} \W^{\#} \colon \Se_r \rightarrow \Se_r$ is the orthogonal projection onto the range of $\W$. Indeed, one may verify that $\Pe^2 = \Pe = \Pe^{\#}$ and $\Pe \W = \W$. 

Now let $u \in \Ce$, let $V \in \R^{r \times \Ce}$ be the matrix satisfying $V y = \sum_{v \in \Ce}{y(v) v}$ for all $y \in \R^\Ce$, and let $X = \left( Vx\right) u^\intercal + u \left( V x \right)^\intercal$. Since $V^\intercal V = M$, we have
\begin{align*}
||X||_F^2 
= \inprod{(Vx) u^{\intercal} + u (Vx)^{\intercal}}{(Vx) u^{\intercal} + u (Vx)^{\intercal}}_F
&= 2 \inprod{Vx}{Vx} \inprod{u}{u} + 2 \inprod{Vx}{u}^2\\
&= 2 \left( x^\intercal M x + (x^\intercal M e_u)^2 \right)\\
&= 2 \left( \lambda + \lambda^2 x(u)^2 \right).
\end{align*}
Moreover, for all $v \in \Ce$ we have
\begin{align*}
(\W^\# X)(v) = \inprod{v v^{\intercal}}{(Vx) u^{\intercal} + u (Vx)^\intercal}_F = 2\inprod{v}{Vx}\inprod{v}{u} 
= 2 (e_v^\intercal M x) \inprod{v}{u}
= 2 \lambda x(v) M(v,u),
\end{align*}
so that $\inprod{\one}{\W^\# X} = 2 \lambda \sum_{v \in \Ce}{ x(v) M(v, u)} = 2 \lambda x^\intercal M e_u = 2 \lambda^2 x(u)$ and 
\begin{align*}
||\W^\# X||^2 = 4 \lambda^2 \sum_{v \in \Ce}{x(v)^2 \inprod{v}{u}^2} 
&= 4 \lambda^2 \left( \alpha^2 \sum_{v \in \Ce}{x(v)^2} + (1-\alpha^2)x(u)^2 \right) \\
&= 4 \lambda^2 \left( \alpha^2 + (1-\alpha^2) x(u)^2  \right).
\end{align*}
Since $\Pe$ is an orthogonal projection, we conclude that
\begin{align*}
2 \left( \lambda + \lambda^2 x(u)^2 \right)
= ||X||^2_F 
\geq ||\Pe X||_F^2
&= X^{\#} \W ( \W^{\#} \W)^{-1} \W^{\#} X\\
&= \frac{1}{1-\alpha^2} (\W^\# X)^{\intercal} \left( I - \frac{\alpha^2}{\alpha^2 n + 1 - \alpha^2} J \right) (\W^\# X)\\
&= \frac{1}{1-\alpha^2}  \left( ||\W^\# X||^2 - \frac{\alpha^2}{\alpha^2 n + 1 - \alpha^2} \inprod{\one}{\W^\# X}^2 \right)\\
&= \frac{1}{1-\alpha^2}  \left(  4 \lambda^2 \left( \alpha^2 + (1-\alpha^2) x(u)^2  \right) - \frac{4 \alpha^2 \lambda^4 x(u)^2}{\alpha^2 n + 1 - \alpha^2}  \right),
\end{align*}
which is equivalent to the desired inequality.

Finally, since $\W^\# \W$ is invertible we have that $n = \rk(\W^{\#} \W) = \rk(\W)$. Also note that the space of symmetric $r \times r$ matrices $\Se_r$ has dimension $\binom{r+1}{2}$. Thus if $n = \binom{r+1}{2}$, then the range of $\W$ is $\Se_r$ and so $\Pe$ is the identity map, giving equality above.
\end{proof}

\begin{remark} \label{rem_schur}
We note that the preceding theorem can also be obtained via an application of the Schur product theorem. Indeed, the projection inequality in the above is equivalent to $||X - \sum_{v \in \Ce}{c(v) v v^\intercal}||^2_F \geq 0$ where $X = (Vx) u^\intercal + u(Vx)^\intercal$ and the coefficients $\{ c(v) : v \in \Ce \}$ are given by the projection of $X$ onto the span of $\{ v v^\intercal: v \in \Ce \}$. Since $X = (Vx + u)(Vx + u)^\intercal - (Vx)(Vx)^\intercal - u u^\intercal$, this inequality can be written in the form $\left|\left|\sum_{v \in \Ce'}{c(v) v v^\intercal} \right|\right|^2_F \geq 0$ where $\Ce' = \Ce \cup \{ Vx, u, Vx + u\}$ and $c(Vx) = c(u) = -1, c(Vx + u) = 1$. Now it suffices to observe that $\left|\left|\sum_{v \in \Ce'}{ c(v) v v^\intercal }\right|\right|^2_F = \sum_{u,v \in \Ce'}{c(u) c(v) \inprod{u}{v}^2} = c^\intercal (M \odot M) c$ where $M = M_{\Ce'}$ is the Gram matrix of $\Ce'$ and $\odot$ is the Hadamard product. Therefore, the desired inequality is equivalent to $c^\intercal (M \odot M) c \geq 0$, which follows from the Schur product theorem applied to $M$.
\end{remark}

\subsection*{Complex equiangular lines}

Another benefit of our approach is that it generalizes to the setting of complex equiangular lines, in which case we have a collection of unit vectors $\Ce$  in $\C^r$ satisfying $| \inprod{u}{v} | = \alpha$ for all $u \neq v \in \Ce$, i.e.\ a spherical $S^1(\alpha)$-code where $S^1(\alpha)$ is the circle of radius $\alpha$ centered at $0$ in the complex plane. As in the real case, it is well known that the Hermitian matrices $\{ v v^{*} : v \in \Ce \}$ form a regular simplex with respect to the Frobenius inner product, so that we have the analogous absolute bound $|\Ce| \leq r^2$. However, unlike in the real case, a complex linear combination of Hermitian matrices need not be Hermitian. Indeed, even for a single Hermitian matrix $H$ we have $(iH)^* = -i H \neq i H$. Therefore, $\{ v v^{*} : v \in \Ce \}$ is capable of spanning the entire space $\C^{r \times r}$ of $r \times r$ complex matrices and so we do not need to restrict ourselves to projecting Hermitian matrices. As such, if we let $u \in \Ce$ and let $x$ be a unit eigenvector of the corresponding Gram matrix $M_\Ce$, then orthogonally projecting $\left( \sum_{v \in \Ce}{x(v) v} \right) u^*$ onto the span of $\{ v v^{*} : v \in \Ce \}$ yields the following result which is analogous to \Cref{main_r}.

\begin{theorem} \label{main_c}
Let $0 < \alpha < 1$ and let $\Ce$ be a spherical $S^1(\alpha)$-code in $\C^r$ with Gram matrix $M = M_{\Ce}$ and $n = |\Ce|$. If $\lambda$ is an eigenvalue of $M$ with corresponding unit eigenvector $x$, then
\[
\lambda \left(\frac{\lambda^2}{ \alpha^2 n + 1 - \alpha^2} - \frac{1-\alpha^2}{\alpha^2} \right) |x(u)|^2 \geq \lambda - \frac{1-\alpha^2}{\alpha^2} 
\]
for all $u \in \Ce$, with equality if $n = r^2$.
\end{theorem}

The proof of \Cref{main_c} is almost identical to (and arguably simpler than) that of \Cref{main_r}, so we omit some of the details. For the reader who is interested in filling in these details, it will be helpful to note that $V^* V = M$ and $\inprod{x y^*}{z w^*}_F = \inprod{x}{z} \overline{\inprod{y}{w}}$ for all $x,y,z,w \in \C^r$.

\begin{proof}[Proof of \Cref{main_c}]
Define the linear map $\W \colon \C^\Ce \rightarrow \C^{r \times r}$ by $\W e_v = v v^*$, so that we have $(\W^\# \W)^{-1} = \frac{1}{1-\alpha^2}\left(I - \frac{\alpha^2}{\alpha^2 n + 1- \alpha^2} J \right)$ and let $\Pe = \W (\W^\# \W)^{-1} \W^\#$ be the corresponding orthogonal projection map. Also let $u \in \Ce$, let $V \in \C^{r \times \Ce}$ be the matrix satisfying $V y = \sum_{v \in \Ce}{y(v) v}$ for all $y \in \C^\Ce$, and let $X = (Vx) u^*$. Then we have
\begin{align*} 
\lambda = ||X||_F^2
\geq || \Pe X||_F^2
&= X^\# \W (\W^\# \W)^{-1} \W^\# X \\
&= \frac{1}{1-\alpha^2}(\W^\# X)^\intercal \left( I - \frac{\alpha^2}{\alpha^2 n + 1 - \alpha^2} J \right) (\W^\# X)\\
&= \frac{1}{1-\alpha^2} \left( ||\W^\# X||^2_F - \frac{\alpha^2}{\alpha^2 n + 1 - \alpha^2} \inprod{\one}{\W^\# X}^2 \right)\\
&= \frac{1}{1-\alpha^2} \left( \lambda^2 \left(\alpha^2 + (1 - \alpha^2)|x(u)|^2 \right) - \frac{\alpha^2}{\alpha^2 n + 1 - \alpha^2} \lambda^4 |x(u)|^2 \right),
\end{align*}
which is equivalent to the desired bound. Moreover, if $n = r^2$ then $ \{ v v^* : v \in \Ce \}$ span $\C^{r \times r}$ and so $\Pe$ is the identity map, giving equality above.
\end{proof}


\begin{remark} \label{rem_SIC}
Recalling that $\odot$ denotes the Hadamard product, we note that projecting $\left( V x \right) \left( V y \right)^*$ onto the span of $\{ v v^* : v \in \Ce \}$ yields the inequality
\[
(1-\alpha^2) \inprod{x}{Mx} \inprod{y}{My} + \frac{\alpha^2}{\alpha^2 n + 1 - \alpha^2} |\inprod{Mx}{My}|^2 \geq \inprod{Mx \odot \overline{Mx}}{My \odot \overline{My}},
\]
which is tight if $n = r^2$. Moreover, the case of equality in the above could be of particular interest because sets of $r^2$ equiangular lines in $\C^r$ are important objects in quantum information theory known as SICs, and \Cref{conj_Zauner} implies that they exist for every $r$, see \Cref{intro_complex} for more information. In particular, it is well known that a set of $r^2$ equiangular lines in $\C^r$ with common angle $\arccos(\alpha)$ must satisfy $\alpha = 1/\sqrt{r+1}$ (see e.g.\ \cite{DGS75}). Therefore, for all vectors $x, y \in \R^\Ce$, the Gram matrix $M$ of a SIC satisfies
\[
r \inprod{x}{Mx} \inprod{y}{My} + \frac{1}{r} |\inprod{Mx}{My}|^2 = (r+1)\inprod{Mx \odot \overline{Mx}}{My \odot \overline{My}}.
\]
\end{remark}

\subsection*{A refinement of the first Welch bound}

It is well known that the relative bound is actually a special case of the first Welch bound, an inequality in coding theory \cite{W74} with applications in signal analysis and quantum information theory, see \cite{EF02, LL14, MM93}. In a more general form, this bound states that for any collection $\Ce$ of $n$ unit vectors in $\C^r$, we have 
\[
\sum_{u, v \in \Ce}{|\inprod{u}{v}|^2} \geq \frac{n^2}{r}.
\]
Moreover, the first Welch bound follows from the inequality $\left|\left| I - \sum_{v \in \Ce}{c(v) v v^{*}} \right|\right|_F^2 \geq 0$ where $c(v) = 1/n$ for all $v \in \Ce$, and so choosing coefficients $\{ c(v) : v \in \Ce \}$ such that $|| I - \sum_{v \in \Ce}{c(v) v v^{*}} ||_F^2$ is minimized would clearly strengthen the bound. Since choosing such coefficients is precisely equivalent to orthogonally projecting the identity $I$ onto the span of $v_1 v_1^*, \ldots, v_n v_n^*$, our method allows us to obtain a refinement of the first Welch bound. 

To state our result, recall that for any Hermitian matrix $M$ with eigenvalues $\lambda_1, \ldots, \lambda_n$ and corresponding unit eigenvectors $u_1, \ldots, u_n$, the Moore--Penrose generalized inverse of $M$ is $M^{\dag} = \sum_{i: \lambda_i > 0}{\frac{1}{\lambda_i} u_i u_i^{\intercal}}$. Also note that $\odot$ denotes the Hadamard product so that the entries of $M \odot \overline{M}$ are given by $\left(M \odot \overline{M}\right)(i,j) = |M(i,j)|^2$.

\begin{theorem} \label{thm_welch}
Let $n \in \N$ and let $\Ce$ be a collection of $n$ unit vectors in $\C^r$ with corresponding Gram matrix $M = M_{\Ce}$. Then 
\[ 
\one^{\intercal} \left(M \odot \overline{M}\right)^{\dag} \one \leq r,
\]
with equality if and only if the $r \times r$ identity matrix is in the span of $\{ v v^* : v \in \Ce \}$. Moreover
\[
\sum_{u,v \in \Ce}{|\inprod{u}{v}|^2}
\geq \frac{n^2}{\one^{\intercal} \left(M \odot \overline{M}\right)^{\dag} \one},
\]
and when $M \odot \overline{M}$ is invertible, we have equality if and only if $\one$ is an eigenvalue of $M \odot \overline{M}$.
\end{theorem}

\begin{proof}
Define the linear map $\W \colon \C^\Ce \rightarrow \C^{r \times r}$ by $W e_v = v v^*$ and consider the adjoint map $\W^\# \colon \C^{r \times r} \rightarrow \C^\Ce$. We have that
\[
(\W^\# \W)(u,v) = \inprod{u}{v} \overline{\inprod{u}{v}} = |\inprod{u}{v}|^2
\]
for all $u, v \in \Ce$, so that $\W^\# \W = M \odot \overline{M}$. Observe that the $r \times r$ identity matrix $I$ satisfies $\W^{\#} I = \one$. Moreover, it is well known (see e.g.\ \cite{BG03}) that $\Pe = \W (\W^{\#} \W)^{\dag} \W^{\#} \colon \C^{r \times r} \rightarrow \C^{r \times r}$ is the orthogonal projection onto the range of $\W$, so that we obtain our first inequality
\[ 
r = ||I||^2_F \geq ||\Pe I||_F^2 = \inprod{I}{\Pe I}_F = I^{\#} \W ( \W^{\#} \W)^{\dag} \W^{\#} I
= \one^{\intercal}  \left(M \odot \overline{M}\right)^{\dag} \one.
\]
Moreover, because $\Pe$ is a projection, we have equality above if and only if $\Pe I = I$, which occurs if and only if $I$ lies in the span of $\{ v v^* : v \in \Ce \}$.

Now we let $\gamma \in \R$ and use the fact that $\Pe I$ is the unique minimizer of the expression $||I - X||_F^2$ over all $X$ in the range of $\W$, in order to conclude that
\begin{align*}
r - \one^{\intercal} \left(M \odot \overline{M}\right)^\dag \one
= \left|\left|I - \Pe I\right|\right|_F^2
\leq  \left|\left|I - \gamma \W \one \right|\right|_F^2
&= I^{\#}I - 2 \gamma I^{\#} \W \one + \gamma^2 \one^{\intercal} \W^{\#} \W \one\\
&= r - 2 \gamma \one^{\intercal} \one + \gamma^2 \one^{\intercal} \left(M \odot \overline{M}\right) \one\\
&= r - 2 n \gamma + \gamma^2 \sum_{u,v \in \Ce}{|\inprod{u}{v}|^2}.
\end{align*}
It is easy to verify that the last expression is minimized at $\gamma \coloneqq n / \sum_{u,v}{|\inprod{u}{v}|^2}$ and thus,
\[
r - \one^{\intercal} \left(M \odot \overline{M}\right)^\dag \one
\leq r - n^2 / \sum_{u,v}{|\inprod{u}{v}|^2},
\]
which is equivalent to the second desired inequality. 

Furthermore, let us assume that $M \odot \overline{M}$ is invertible, so that we have $\Pe I = \W (\W^{\#} \W)^{\dag} \W^{\#} I = \W \left(M \odot \overline{M}\right)^{-1} \one$. By the uniqueness of $\Pe I$, the previous inequality is tight if and only if $\Pe I = \gamma \W \one$. Thus if $\Pe I = \gamma  \W \one$, we conclude
\[
\gamma \left(M \odot \overline{M}\right) \one = \gamma \W^{\#} \W \one 
= \W^{\#} \Pe I
= \W^{\#} \W \left(\W^\# \W\right)^{-1} \W^\# I
= \one,
\] 
so that $\one$ is an eigenvector of $M \odot \overline{M}$. On the other hand, if $\one$ is an eigenvector of $M \odot \overline{M}$, then the corresponding eigenvalue $\lambda$ satisfies 
\[
\lambda = \frac{1}{n} \one^\intercal \left(M \odot \overline{M}\right) \one  = \frac{1}{\gamma}
\]
and therefore, $\left(M \odot \overline{M}\right)^{-1} \one = \gamma \one$. Applying $\W$ now yields
\[
\Pe I =  \W \left(\W^\# \W \right)^{-1} \W^\# I = \W \left(M \odot \overline{M}\right)^{-1} \one = \gamma \W \one,
\]
as desired.
\end{proof}

\begin{remark} \label{rem_relative}
Note that if $\Ce$ is a spherical $\{\alpha, -\alpha\}$-code, then we have already observed that $\left(M \odot \overline{M}\right)^{-1} = \frac{1}{1-\alpha^2} \left( I - \frac{\alpha^2}{\alpha^2 n + 1 - \alpha^2 } J\right) $ and so the inequality $\one^\intercal \left(M \odot \overline{M}\right)^{\dag} \one \leq r$ in the above is equivalent to the relative bound. 
\end{remark}

\begin{remark}
If one is interested in efficiently computing the term $\one^\intercal \left(M \odot \overline{M}\right)^{\dag} \one$ in \Cref{thm_welch} when $M \odot \overline{M}$ is invertible, then this can be done without having to determine $\left(M \odot \overline{M}\right)^{\dag} = \left(M \odot \overline{M}\right)^{-1}$. Indeed, it suffices to find a solution $x \in \R^n$ to $\left(M \odot \overline{M}\right) x = \one$, since then $\one^\intercal x = \one^\intercal \left(M \odot \overline{M}\right)^{-1} \one$.
\end{remark}

\subsection*{Connection to Delsarte's linear programming method}

We conclude this section by discussing the connection between our projection based method and the classical linear programming approach of Delsarte, Goethels, and Seidel \cite{DGS75, DGS77} via Gegenbauer polynomials. 
Given a function $f \colon \R \rightarrow \R$ and a matrix $M \in \R^{n \times n}$, we let $f(M) \in \R^{n \times n}$ denote the matrix resulting from applying $f$ to each entry of $M$, i.e.\ $f(M)(i,j) = f(M(i,j))$. Observe that if we let $\Ce = \{v_1, \ldots, v_n\}$ be a spherical $\{\alpha, - \alpha\}$-code in $\R^r$, then with respect to the Frobenius inner product, projecting each $v_i v_i^{\intercal}$ onto the orthogonal complement of the identity $I$ results in a collection of matrices $\Ce' = \{ v_1 v_1^{\intercal} - \frac{1}{r}I, \ldots, v_n v_n^{\intercal} - \frac{1}{r}I \}$ which has the Gram matrix $M_{\Ce'} = G^r_2(M_{\Ce})$, where $G^r_2(x) = x^2 - \frac{1}{r}$ is a scaled version of the second Gegenbauer polynomial of the $(r-1)$-sphere. 

While the Gegenbauer polynomials and their properties are usually derived using orthogonal spaces of harmonic homogeneous polynomials, see e.g.\ \cite{DGS77}, one may equivalently obtain them via an appropriate generalization of the previous argument to symmetric tensors, see Ehrentraut and Muschik \cite{EM98} for more information. Indeed, a symmetric matrix which is orthogonal to the identity can be viewed as a traceless symmetric 2-tensor and in general, there is an isomorphism between the inner product spaces of traceless symmetric $k$-tensors and harmonic homogeneous polynomials of degree $k$. 
While this approach to spherical harmonics via symmetric tensors does not seem to be common in the mathematics literature, it is more well known in the physics literature, see e.g.\ \cite{P70, H87, EM98, ZZ03, CHS04}.

In addition to being generalizations of the Chebyshev and Legendre polynomials, the significance of Gegenbauer polynomials goes back to a classical result of Schoenberg \cite{S42}, who proved that a function $f \colon [-1,1] \rightarrow \R$ has the property that $f(M_{\Ce})$ is positive semidefinite for any finite set of unit vectors $\Ce$ in $\R^r$ if and only if $f$ is a nonnegative linear combination of Gegenbauer polynomials. Indeed, the inequality $\one^{\intercal} f(M_{\Ce}) \one \geq 0$ with $f$ chosen appropriately\footnote{If $f$ is a nonnegative linear combination of Gegenbauer polynomials such that $f(\inprod{v_i}{v_j}) \leq -1$ for all $i \neq j$, then $0 \leq \one^\intercal f(M_\Ce) \one \leq n f(1) - n(n-1)$, so that $n \leq f(1) + 1$. Therefore, to get the best bound on $n$, we must choose $f$ subject to the given constraints so as to minimize $f(1)$. This minimization can be done using linear programming.} underlies the linear programming approach of Delsarte, Goethels, and Seidel, which has the benefit of only requiring an upper bound on the off-diagonal entries of $f(M_\Ce)$ in order to obtain an upper bound on $n$. Our approach could therefore be seen as a more refined use of the Frobenius inner product geometry of (symmetric) matrices/tensors in order to obtain sharp bounds on the largest eigenvalue of $M_{\Ce}$ which then yield sharp upper bounds on $n$. Although one could obtain the same results as we do by working with polynomials with respect to a certain inner product, we believe that using symmetric tensors is both simpler and more suggestive. The drawback of our approach seems to be that it require some control over the (generalized) inverse of $f(M_\Ce)$ in order to be effective.

\section{Equiangular lines in \texorpdfstring{$\R^r$}{R\^{}r}} \label{section_r}

In this section we will prove our main theorems regarding real equiangular lines. Our methods build on some of the ideas appearing in \cite{BDKS18, JP20, JTYZZ21}. We will first give an outline for proving \Cref{thm_bound1_r}. To this end, we begin by establishing some basic definitions and observations.

Note that the case $\alpha = 0$ is trivial since it corresponds to orthogonal vectors, so we can assume that $\alpha > 0$. For any family of equiangular lines $\Le$, we say that a spherical $\{\alpha, -\alpha\}$-code $\Ce$ \emph{represents} $\Le$ if there exists a bijection $\ell \colon \Ce \rightarrow \Le$ such that $v \in \ell(v)$ for all $v \in \Ce$. As noted in previous sections, choosing a unit vector along each line  $l \in \Le$ always yields a spherical $\{\alpha,-\alpha\}$-code which represents $\Le$, where $\alpha$ is the cosine of the common angle. Therefore, instead of working with a family of lines directly, we will always consider some spherical $\{\alpha, -\alpha\}$-code $\Ce$ which represents it. 

Our first lemma shows that we can assume without loss of generality that an eigenvector corresponding to the largest eigenvalue of the corresponding Gram matrix $M_\Ce$ has no negative coordinates. It is proved by negating a suitable subset of vectors. Indeed, this operation is known in the literature as \emph{switching}, and it is useful for us because it negates coordinates of eigenvectors of $M_{\Ce}$ without affecting eigenvalues.

\begin{lemma} \label{lem_nonnegative_eigenvector}
For any family $\Le$ of equiangular lines, there exists a spherical $\{\alpha,-\alpha\}$-code $\Ce$ which represents $\Le$ such that an eigenvector corresponding to the largest eigenvalue of the Gram matrix $M_\Ce$ has no negative coordinates.
\end{lemma}
\begin{proof}
Let $\Ce$ be a spherical $\{\alpha, -\alpha\}$-code which represents $\Le$ and let $x$ be an eigenvector of $M_\Ce$ corresponding to $ \lambda_1(M_\Ce)$. Let $S = \{u : x(u) < 0 \}$ denote the negative coordinates of $x$ and let $D$ be the $\Ce \times \Ce$ diagonal matrix given by $D(u,u) = -1$ if $u \in S$ and $D(u,u) = 1$ otherwise, so that $Dx$ has no negative coordinates. Now let $\Ce' = (\Ce \backslash S) \cup \{ -u: u \in S \}$ be the spherical $\{\alpha, -\alpha\}$-code obtained from $\Ce$ by negating the vectors corresponding to $S$, and note that it clearly also represents $\Le$. Since its Gram matrix satisfies $M_{\Ce'} = D M_{\Ce} D$, it is easy to check that $Dx$ is an eigenvector of $M_{\Ce'}$ corresponding to $\lambda_1(M_{\Ce}) = \lambda_1(M_{\Ce'})$.
\end{proof}

As previously discussed, we can associate a corresponding graph $G_{\Ce}$ to any spherical $\{\alpha, -\alpha\}$-code $\Ce$ with an edge for each pair of vectors having a negative inner product. More precisely, we define $G = G_{\Ce}$ to have vertex set $V(G) = \Ce$ and edge set $E(G) = \{uv : u,v \in \Ce \text{ and } \inprod{u}{v} = -\alpha \}$. Moreover, letting $A = A_G$ be the adjacency matrix of $G$ and $M = M_{\Ce}$ be the Gram matrix of $\Ce$, we will make repeated use of the fact that
\begin{equation} \label{eq_gram_adjacency}
M = (1-\alpha)I + \alpha J - 2 \alpha A.
\end{equation}

Before giving the outline of our approach for proving \Cref{thm_bound1_r}, we also observe that the upper bound of $\binom{1/\alpha^2 - 1}{2}$ on the size of a spherical $\{\alpha, -\alpha\}$-code follows immediately from the assumption that the largest eigenvalue of its Gram matrix is at most $\frac{1-\alpha^2}{2\alpha^2}$, and we also characterize when equality occurs. 

\begin{lemma} \label{lem_lambda_small}
Let $0 < \alpha < 1$ and let $\Ce$ be a spherical $\{\alpha, -\alpha\}$-code. If the largest eigenvalue $\lambda_1$ of the Gram matrix $M = M_\Ce$ satisfies $\lambda_1 \leq \frac{1-\alpha^2}{2\alpha^2}$, then
\[
|\Ce| \leq \binom{1/\alpha^2 - 1}{2}
\]
with equality only if the span of $\Ce$ has dimension $1/\alpha^2 - 2$.
\end{lemma}
\begin{proof}
Let $n = |\Ce|$ and $r = \rk(M)$. First observe that
\begin{equation} \label{eq_trace_square}
\sum_{i=1}^{r}{\lambda_i(M)^2}  =   \tr(M^2) = (\alpha^2 n + 1 - \alpha^2) n.
\end{equation}
Since $\lambda_1 \leq \frac{1-\alpha^2}{2\alpha^2}$, the desired bound immediately follows from the fact that 
\begin{equation} \label{eq_lambda_small}
\sum_{i=1}^{r}{\lambda_i(M)^2}  \leq \lambda_1 \tr(M) = \lambda_1 n.
\end{equation}
Moreover, observe that equality occurs only if $\lambda_i(M) = \frac{1-\alpha^2}{2\alpha^2}$ for all $1 \leq i \leq r$. Note also that $r$ is the dimension of the span of $\Ce$ and by changing bases, we may assume that $\Ce \subset \R^{r}$, so that if we let $V \in \R^{r \times \Ce}$ be the matrix given by $Vy = \sum_{v \in \Ce}{y(v) v}$, then $M = V^\intercal V$ has the same nonzero eigenvalues as the $r \times r$ matrix $V V^\intercal = \sum_{v \in \Ce}{v v^\intercal}$. Therefore, $n = \binom{1/\alpha^2 - 1}{2}$ implies that $ \sum_{v \in \Ce}{v v^\intercal} = \frac{1-\alpha^2}{2\alpha^2} I$ where $I$ is the $r \times r$ identity matrix and by taking the trace, we conclude that $\binom{1/\alpha^2 - 1}{2} = n = \frac{1-\alpha^2}{2\alpha^2} r$ and so $r = 1/\alpha^2 - 2$.
\end{proof}

\begin{proof}[Outline of the proof of \Cref{thm_bound1_r}] If $\alpha > 1/7$, then $N_\alpha^\R(r) \leq \max \left( \binom{1/\alpha^2 -1}{2}, 2r  \right )$ follows from the results of Neumann and Lemmens and Seidel \cite{LS73}, and Cao, Koolen, Lin, and Yu \cite{CKLY22}. Now we may assume that $\alpha \leq 1/7$ and let $\Ce$ be a spherical $\{\alpha, -\alpha\}$-code in $\R^r$ with $n = |\Ce| = N^\R_\alpha(r)$ as given by \Cref{lem_nonnegative_eigenvector}, so that the corresponding Gram matrix $M = M_\Ce$ has largest eigenvalue $\lambda_1$ with a corresponding unit eigenvector $x$ having no negative coordinates. Note that $2r - \frac{(1+\alpha)^2}{8 \alpha^2}$ is larger than $ \binom{1/\alpha^2 - 1}{2} $ precisely when $r > r_\alpha = \frac{1}{2} \binom{1/\alpha^2 - 1}{2} + \frac{(1+\alpha)^2}{16 \alpha^2}$ and so our goal will be to bound $n$ from above by $\max\left( \binom{1/\alpha^2 - 1}{2}, 2r - \frac{(1+\alpha)^2}{8 \alpha^2}\right)$. As such, we may also assume without loss of generality that $n \geq \binom{1/\alpha^2 - 1}{2}$. Finally, if $\lambda_1 \leq \frac{1-\alpha^2}{2\alpha^2}$ then we obtain the first desired upper bound of $\binom{1/\alpha^2 - 1}{2}$ using \Cref{lem_lambda_small}. 

Otherwise, we will have $\lambda_1 > \frac{1-\alpha^2}{2\alpha^2}$ in which case \Cref{main_r} will imply that $\lambda_1$ is actually much larger (nearly $\alpha n$). More precisely, we will show that $x(u) > \alpha \sqrt{n} / \left(\lambda_1 + \frac{1-\alpha^2}{4\alpha^2} \right)$ for any $u \in \Ce$, so that together with \eqref{eq_gram_adjacency}, we will conclude\footnote{Intuitively, our argument works because it is implicitly showing that $x$ must be close to the normalized all ones vector $\frac{1}{\sqrt{n}} \one$.} that the maximum degree $\Delta$ of the corresponding graph satisfies $\Delta < \frac{1 + 3 \alpha^2 }{8 \alpha^3}$. Note that the corresponding graph will have second eigenvalue $\frac{1-\alpha}{2\alpha}$ and so this can be seen as an Alon-Boppana-type theorem. We will then derive other new variants of the Alon--Boppana theorem which we will use in a bootstrapping argument in order to obtain the improved bound $\Delta \leq \left( \frac{1-\alpha}{2\alpha} \right)^2$. To accomplish this, we will first use \Cref{lem_beta_lower} to prove \Cref{lem_bootstrap}, which will get us $99\%$ of the way towards this bound by establishing that $\Delta < \left(\frac{1}{2\alpha}\right)^2$. More precisely, we will first use the fact that $n \geq \binom{1/\alpha^2 - 1}{2}$ to conclude that $\Delta < n/20$.  We will then apply \Cref{lem_beta_lower} with $t = \left \lceil  \frac{n}{4 \Delta} \right \rceil^2$ to obtain $\Delta < \left(\frac{n}{4} \right)^{2/3}$. Next, we will apply \Cref{lem_beta_lower} again with $t = \Delta$ to conclude that $\Delta < \left(\frac{3(1-\alpha)}{4\alpha} \right)^2$ and then apply \Cref{lem_beta_lower} one final time with $t = \Delta$ to obtain the $99\%$ bound $\Delta < \left(\frac{1}{2\alpha}\right)^2$. To get us all the way towards the desired bound, we first note that if $n < r + \frac{1}{16\alpha^4}  + \frac{1}{4\alpha^2} + 1$, then we are done and otherwise, the second eigenvalue $\frac{1-\alpha}{2\alpha}$ in the corresponding graph has multiplicity at least $\frac{1}{16\alpha^4}  + \frac{1}{4\alpha^2}$. In \Cref{lem_bootstrap_optimal}, we then use \Cref{lem_optimal_multiplicity} to conclude that if $\Delta  > \left( \frac{1-\alpha}{2\alpha} \right)^2$, then this multiplicity is at most $\Delta(\Delta + 1)$, leading to a contradiction and allowing us to finally establish the desired sharp Alon--Boppana-type bound $\Delta \leq \left( \frac{1-\alpha}{2\alpha} \right)^2$.

Now via \eqref{eq_gram_adjacency}, it will follow that $\lambda_1 \geq \frac{\one^\intercal M \one}{n} \geq 1-\alpha +  \alpha n - 2 \alpha \Delta$, so that using our sharp bound on $\Delta$, we may conclude that $\lambda_1$ will indeed be almost as large as $\alpha n$, which is equivalent to 
the matrix $L = M - \lambda_1 x x^\intercal$ having a small Frobenius norm. On the other hand, $L$ has a large trace (roughly $n$) and so, we will conclude that $L$ has large rank (roughly $n/2$). Intuitively, this will follow because the conditions of $L$ having a small Frobenius norm and a large trace are equivalent to $L$ being close in Frobenius distance to the identity matrix, which has full rank. Since the rank of $L$ is at most $r-1$, we will thus obtain the second desired bound $n \leq 2r - \frac{(1+\alpha)^2}{8 \alpha^2}$, completing the argument.
\end{proof}

\begin{remark}
The proof of \Cref{thm_bound3_r} will be almost identical to that of \Cref{thm_bound1_r}, except that we will apply the bootstrapping argument with a stronger version of the Alon-Boppana theorem, as done in \cite{JP20}. Furthermore, \Cref{thm_bound4_r} will follow by combining our new bound on the maximum degree with the second eigenvalue multiplicity argument of \cite{JTYZZ21}.
\end{remark}

In order to implement the arguments in the outline given above, we now state some more useful definitions and results. 
The following lemma shows that a matrix with small Frobenius norm and large trace must have high rank. It is a quantified form of the idea that if a matrix is close to the identity matrix, then its rank must also be close to that of the identity matrix. Note that this result, whose proof is a simple application of Cauchy--Schwarz, 
was an important ingredient of our arguments in \cite{BDKS18} and has many other applications, see e.g.\ the survey of Alon \cite{A09}. 

\begin{lemma} \label{lem_schnirelman}
For any Hermitian matrix $L$, we have
\[ \tr(L)^2 \leq \tr(L^2) \rk(L). \]
\end{lemma}
\begin{proof}
Let $r = \rk(L)$ and note that $\tr(L) = \sum_{i=1}^{r}{\lambda_i(L)}$ and $\tr(L^2) = \sum_{i=1}^{r}{\lambda_i(L)^2}$. The result now follows via Cauchy--Schwarz.
\end{proof}
\begin{remark} \label{rem_identity}
In keeping with the theme of this paper, we note that one can alternatively prove \Cref{lem_schnirelman} using the Frobenius inner product, under the condition that $L$ is positive semidefinite (which will indeed be the case in our applications). Indeed, if $L$ is $n \times n$ with rank $r$, then there exists an $r \times n$ matrix $V$ such that $L = V^* V$ and so if we let $I$ be the $r \times r$ identity matrix, then applying Cauchy--Schwarz with the Frobenius inner product yields
\[
\tr(L)^2 = 
\tr(V V^*)^2 = \inprod{V V^*}{I}_F^2 \leq || V V^* ||_F^2 || I ||_F^2
= \tr(V V^*V V^*) \tr(I)
= \tr(L^2) r.
\]
\end{remark}

Note that \Cref{lem_schnirelman} can be used to obtain an upper bound on the size of a spherical $\{\alpha, -\alpha\}$-code $\Ce$ in $\R^r$, in terms of $r$, $\alpha$, and the average degree $\overline{d}$ of the corresponding graph. Indeed, if we let $M = M_{\Ce}$ be the corresponding Gram matrix, then applying \Cref{lem_schnirelman} to the matrix $M - \alpha J$ yields
\begin{equation} \label{eq_simple_bound}
|\Ce| \leq \left( 1 + \left( \frac{2 \alpha }{1 - \alpha} \right)^2 \overline{d} \right) (r+1).
\end{equation} 
However, in the proof of \Cref{thm_bound1_r}, we will prove a slightly more precise bound, which requires an assumption on how large the maximum degree $\Delta$ is. We now turn our attention toward obtaining such an upper bound on $\Delta$. We first apply \Cref{main_r} to obtain a weaker upper bound on $\Delta$, provided that $\lambda_1 > \frac{1-\alpha^2}{2\alpha^2}$ and $n$ is sufficiently large. We will later apply bootstrapping arguments in order to sufficiently improve this bound.

\begin{lemma} \label{lem_degree}
Let $0 < \alpha < 1$ and let $\Ce$ be a spherical $\{\alpha, -\alpha\}$-code with $|\Ce| \geq \binom{1/\alpha^2 - 1}{2}$. If the largest eigenvalue $\lambda_1 = \lambda_1(M)$ of the Gram matrix $M = M_\Ce$ satisfies $\lambda_1 > \frac{1-\alpha^2}{2\alpha^2}$ and has a corresponding eigenvector $x$ with no negative coordinates, then the maximum degree of the corresponding graph $G = G_{\Ce}$ satisfies
\[
\Delta_G < \frac{1 + 3\alpha^2 - 4 \alpha^3}{8 \alpha^3}. 
\]
\end{lemma}
\begin{proof}
By normalizing, we may assume that $x$ is a unit vector. Now let $n = |\Ce|$, let $u \in \Ce$, and observe that \Cref{main_r} implies 
\begin{align*}
\frac{x(u)^2}{n} \geq \frac{1 - \frac{1-\alpha^2}{2 \alpha^2 \lambda_1}}{ \frac{n \lambda_1^2  }{\alpha^2 n  + 1 - \alpha^2} - \frac{1-\alpha^2}{2 \alpha^2} n }
\eqqcolon Q,
\end{align*} 
where the numerator of $Q$ is positive by assumption and the denominator of $Q$ is positive due to \eqref{eq_trace_square} and \eqref{eq_lambda_small}. Since $x$ has no negative coordinates, we conclude that $x(u) \geq \sqrt{Q n}$ for all $u \in \Ce$ and it follows from \eqref{eq_gram_adjacency} that
\begin{align*}
(\lambda_1 - 1 + \alpha) \sqrt{Q n} \leq (\lambda_1 - 1 + \alpha)x(u) 
= ( (M - (1-\alpha)I)x)(u)
&= ( (\alpha J - 2\alpha A)x )(u)\\
&= \alpha \inprod{\one}{x} - 2 \alpha \sum_{v \in N_G(u)}{x(v)}\\
&\leq \alpha \sqrt{n} - 2 \alpha \sqrt{Q n} \cdot d_G(u),
\end{align*}
so we have $d_G(u) \leq  \frac{1}{2\sqrt{Q}} - \frac{\lambda_1}{2 \alpha} + \frac{1 - \alpha}{2 \alpha}$. It now remains to verify that $Q$ satisfies the lower bound
\begin{equation} \label{eq_Q}
Q > \frac{\alpha^2}{ \lambda_1^2 \left(1  +  \frac{1-\alpha^2}{2\alpha^2 \lambda_1} \right) }.
\end{equation}
Indeed, if \eqref{eq_Q} holds then we have $\frac{1}{\sqrt{Q}} < \frac{\lambda_1}{\alpha} \sqrt{1 + \frac{1-\alpha^2}{2\alpha^2 \lambda_1} } < \frac{\lambda_1}{\alpha} + \frac{1 - \alpha^2}{4 \alpha^3}$ and thus
\[ 
d_G(u) < \frac{1-\alpha^2}{8 \alpha^3} + \frac{1-\alpha}{2\alpha} = \frac{1 + 3\alpha^2 - 4 \alpha^3}{8 \alpha^3}. 
\]
Now observe that by cross multiplying and collecting terms, \eqref{eq_Q} is equivalent to
\[
\left(\frac{1-\alpha^2}{2 \alpha^2} \right)^2 - \frac{1-\alpha^2}{2}n  < \lambda_1^2 \left( 1 - \frac{\alpha^2 n}{\alpha^2 n + 1 - \alpha^2}  \right) 
= \frac{1-\alpha^2 }{\alpha^2 n + 1 - \alpha^2 } \lambda_1^2,
\]
so if we define $f(n) = \left( \frac{1 - \alpha^2}{4 \alpha^4} - \frac{n}{2} \right) \left( \alpha^2 n + 1 - \alpha^2 \right)$, then this is also equivalent to $f(n) < \lambda_1^2$. By computing its derivative, it is easy to see that $f(n)$ is decreasing for $n \geq \frac{1}{2} \binom{1/\alpha^2 - 1}{2}$ and we have $n \geq \max\left( \binom{1/\alpha^2 - 1}{2}, 0 \right) \geq \frac{1}{2} \binom{1/\alpha^2 - 1}{2}$, so $f$ is maximized when $n = \max\left( \binom{1/\alpha^2 - 1}{2}, 0 \right)$. Now \eqref{eq_Q} follows by observing that for both $n = \binom{1/\alpha^2 - 1}{2}$ and $n = 0$, we have $f(n) = \left( \frac{1-\alpha^2}{2 \alpha^2} \right)^2 < \lambda_1^2$. 
\end{proof}

\begin{remark} \label{rem_direct_lambda}
Note that \Cref{lem_degree}, when used together with \eqref{eq_simple_bound} and \Cref{lem_lambda_small}, already implies that $N^\R_\alpha(r) \leq \max \left( \binom{1/\alpha^2 - 1}{2}, O\left(r / \alpha \right) \right)$. Moreover, instead of using \Cref{lem_degree}, we note that a more precise application of \Cref{main_r} can be used to show that the $O\left(r / \alpha \right)$ term is at most $\frac{1+\alpha}{2\alpha}r + O(1/\alpha^3)$ (see \Cref{section_c} for the analogous argument in the complex case). Such a bound was recently conjectured to be obtainable from the second level of the Lasserre hierarchy by de Laat, Keizer, and Machado \cite{LKM22}, so it would be interesting to determine if this conjecture follows from our arguments.
\end{remark}

To improve on the $O(r/\alpha)$ bound mentioned in the preceding remark, we will strengthen the degree bound of \Cref{lem_degree} by a factor of $\alpha$ by using a bootstrapping argument via some new Alon--Boppana-type theorems. To this end, it will be convenient to introduce the following spectral graph parameter. For any graph $G$ with adjacency matrix $A = A_G$, we define
\[
\beta(G) = \max_{\substack{x \in \R^{V(G)} \backslash \{0\} \\ x \perp \one}}{\frac{x^{\intercal} A x}{x^{\intercal} x}},
\]
i.e. the maximum Rayleigh quotient over all vectors orthogonal to the all ones vector $\one$. Note that via the Courant min-max principle, the second largest eigenvalue satisfies $\lambda_2(G) \leq \beta(G)$ with equality when $G$ is regular, so that $\beta$ can be seen as an extension to all graphs of the second largest eigenvalue of a regular graph. Using \eqref{eq_gram_adjacency}, we immediately obtain the following characterization of $\beta$ for the graph corresponding to a spherical $\{\alpha, -\alpha\}$-code.

\begin{lemma} \label{lem_beta_upper}
Let $0 < \alpha < 1$ and let $\Ce$ be a spherical $\{\alpha, -\alpha\}$-code in $\R^r$ with corresponding graph $G = G_{\Ce}$. Then 
\[
\beta(G) \leq \frac{1-\alpha}{2\alpha}
\]
with equality whenever $|\Ce| \geq r + 2$.
\end{lemma}
\begin{proof}
Let $M$ and $A$ be the corresponding Gram and adjacency matrices, respectively and let $x \in \R^\Ce \backslash \{0\}$ be such that $x \perp \one$. Then $Jx = 0$ and so via \eqref{eq_gram_adjacency} we have
\[
0 \leq x^{\intercal} M x 
= (1-\alpha) x^{\intercal} I x - 2 \alpha x^{\intercal} A x,
\]
so that $\frac{x^{\intercal}Ax}{x^{\intercal}x} \leq \frac{1-\alpha}{2\alpha}$. Since $x$ was arbitrary, we conclude that $\beta(G) \leq \frac{1-\alpha}{2\alpha}$ as desired.

Moreover, if $|\Ce| \geq r+2$ then since $\rk(M) \leq r$, the nullspace of $M$ has dimension at least 2 and must therefore contain a nonzero vector $y$ which is orthogonal to $\one$. Using \eqref{eq_gram_adjacency}, we conclude
\[
0 = y^{\intercal} M y
= (1-\alpha) y^{\intercal} y - 2 \alpha y^{\intercal} A y,
\]
which implies that $\beta(G) \geq \frac{y^{\intercal} A y}{y^{\intercal} y} = \frac{1-\alpha}{2\alpha}$.
\end{proof}

Having established an upper bound on $\beta$, we now state a variant of the Alon--Boppana theorem which gives a lower bound on $\beta$ for any graph $G$ with a given maximum degree. In particular, we follow Friedman and Tillich's approach \cite{FT05} of starting with an eigenvector corresponding to a subgraph $H \subseteq G$ and then ``projecting out the constant''  in order to make this eigenvector orthogonal to the all ones vector.

\begin{lemma} \label{lem_friedman}
Let $G$ be a graph on $n$ vertices with maximum degree $\Delta = \Delta(G)$.
For any subgraph $H \subseteq G$ with $x$ being a unit eigenvector corresponding to $\lambda_1(H)$, we have
\[
\beta(G) \geq \lambda_1(H) - \frac{2 \Delta}{n} \inprod{x}{\one}^2.
\]
\end{lemma}
\begin{proof}
First note that by the Perron--Frobenius theorem, all coordinates of $x$ are nonnegative. Now observe that we may extend $x$ to $\R^{V(G)}$ by defining $x(u) = 0$ for all $u \notin V(H)$, so that we conclude $x^{\intercal} A x \geq \lambda_1(H)$. Now we define $y = x - \frac{\inprod{x}{\one}}{n}\one$, i.e.\ the projection of $x$ onto the orthogonal complement of $\one$. The desired bound now follows by computing that $y^{\intercal} y = ||x||^2 - \frac{\inprod{x}{\one}^2}{n} \leq 1$ and
\begin{align*}
y^{\intercal} A y 
= x^{\intercal} A x - 2\frac{\inprod{x}{\one}}{n} \one^{\intercal} A x + \frac{\inprod{x}{\one}^2}{n^2} \one^{\intercal} A \one
&\geq \lambda_1(H) - 2\frac{\inprod{x}{\one}}{n} \sum_{u \in V(G)}{x(u) d_G(u) }\\
&\geq \lambda_1(H) - \frac{2 \Delta}{n}\inprod{x}{\one}^2. \qedhere
\end{align*}
\end{proof}

The previous lemma suggests finding a small subgraph $H$ with $\lambda_1(H)$ large and indeed, in the proof of the usual Alon--Boppana theorem one takes $H$ to be a ball centered at a vertex. Indeed, for the proof of \Cref{thm_bound1_r} we take $H = K_{1,t}$, i.e.\ the graph consisting of a vertex connected to $t$ other vertices, and so we will first determine $\lambda_1(K_{1,t})$ and its corresponding unit eigenvector.

\begin{lemma} \label{lem_star}
For all $t \in \N$, we have $\lambda_1(K_{1,t}) = \sqrt{t}$ and the corresponding unit eigenvector $x$ satisfies $\inprod{\one}{x} = (\sqrt{t} + 1)/\sqrt{2} $.
\end{lemma}
\begin{proof}
Let $H = K_{1,t}$ so that we have $V(H) = T \cup \{ w  \}$ and $E(H) = \{w v : v \in T \}$. If we let $A = A_H$ be the corresponding adjacency matrix, then it is a straightforward calculation to verify that the maximum of $x^\intercal A x = 2 x(w) \sum_{v \in T}{x(v)}$ over all unit vectors $x \in \R^{V(H)}$ occurs when $x(v) = 1/\sqrt{2t}$ for $v \in T$ and $x(w) = 1/\sqrt{2}$, in which case we have $x^{\intercal} A x = \sqrt{t}$ and $\inprod{\one}{x} = (\sqrt{t} + 1)/\sqrt{2} $.
\end{proof}

\noindent We now combine \Cref{lem_friedman} and \Cref{lem_star} to obtain the desired lower bound on $\beta$.

\begin{lemma} \label{lem_beta_lower}
Let $G$ be a graph on $n$ vertices with maximum degree $\Delta = \Delta_G$. For all $t \in \N$ with $t \leq \Delta$, we have
\[
\beta(G) \geq \sqrt{t} - \frac{\Delta }{n}\left(\sqrt{t} + 1\right)^2.
\]
\end{lemma}
\begin{proof}
Let $w \in V(G)$ be a vertex having degree $d_G(v) = \Delta$ and let $H$ be a subgraph consisting of $w$ together with $t$ of its neighbors. Applying \Cref{lem_star} and \Cref{lem_friedman} with $H$, we have
\[
\beta(G) \geq \lambda_1(H) - \frac{2 \Delta}{n}\frac{\left(\sqrt{t} + 1 \right)^2}{2} \geq \sqrt{t} - \frac{\Delta}{n} \left(\sqrt{t} + 1\right)^2. \qedhere
\]
\end{proof}

Using \Cref{lem_beta_upper} and \Cref{lem_beta_lower}, we now derive the bootstrapping argument which improves the upper bound on the maximum degree of \Cref{lem_degree} by a factor of $\alpha$.

\begin{lemma} \label{lem_bootstrap}
Let $0 < \alpha \leq 1/7$ and let $\Ce$ be a spherical $\{\alpha, -\alpha\}$-code with $|\Ce| \geq \binom{1/\alpha^2 - 1}{2}$. If the largest eigenvalue $\lambda_1$ of the Gram matrix $M = M_\Ce$ satisfies $\lambda_1 > \frac{1-\alpha^2}{2\alpha^2}$ and has a corresponding eigenvector $x$ with no negative coordinates, then the maximum degree of the corresponding graph $G = G_{\Ce}$ satisfies
\[
\Delta_G < \frac{1}{4 \alpha^2}.
\]
\end{lemma}
\begin{proof}
Let $n = |\Ce|$. Using \Cref{lem_degree}, we have $\Delta = \Delta_G < \frac{1 + 3\alpha^2}{8\alpha^3}$. Since $n \geq \binom{1/\alpha^2 - 1}{2} > \frac{1 - 3 \alpha^2}{2 \alpha^4}$ and $\alpha \leq 1/7$, we conclude that 
\[
\frac{\Delta}{n} \leq  \frac{(1+3\alpha^2) \alpha}{(1 - 3\alpha^2) 4} < \frac{ \alpha }{ 3 } < \frac{1}{20}.
\]
We will now show that $\Delta \leq \left \lceil \left( \frac{n}{4 \Delta} \right)^2 \right \rceil - 1$, so suppose for sake of contradiction that $\Delta \geq \left \lceil \left( \frac{n}{4 \Delta} \right)^2 \right \rceil$. Let $s = \left \lceil \left( \frac{n}{4 \Delta} \right)^2 \right \rceil$ and observe that $\sqrt{s} < \sqrt{ \left( \frac{n}{4 \Delta} \right)^2 + 1} < \frac{n}{4 \Delta} + \frac{1}{10}$, so that applying \Cref{lem_beta_upper} and \Cref{lem_beta_lower} with $t = s$ yields
\begin{align*}
\frac{1 - \alpha}{2 \alpha} \geq \beta(G) 
\geq \sqrt{s} - \frac{\Delta}{n} \left( \sqrt{s} + 1 \right)^2
\geq \frac{n}{4 \Delta}  - \frac{\Delta}{n} \left( \frac{n^2}{16 \Delta^2 } + \frac{11n }{20 \Delta}  + \frac{121 }{100} \right)
&= \frac{3n}{16 \Delta}  - \frac{11}{20}  - \frac{121 \Delta}{100 n} \\
&> \frac{9}{16 \alpha}  - \frac{11}{20}  - \frac{121 \alpha}{300},
\end{align*}
which contradicts the assumption that $\alpha \leq 1/7$.

Now that we have established $\Delta \leq \left \lceil \left( \frac{n}{4 \Delta} \right)^2 \right \rceil - 1 < \left( \frac{n}{4 \Delta} \right)^2$, we have $\sqrt{\Delta} < \frac{n^{1/3}}{2^{2/3}}$ and so we apply \Cref{lem_beta_upper} and \Cref{lem_beta_lower} once again, but this time with $t = \Delta$, in order to conclude that
\begin{align*}
\frac{1 - \alpha}{2 \alpha} \geq \beta(G) 
\geq \sqrt{\Delta} \left( 1 -  \frac{\sqrt{\Delta}}{n} \left( \sqrt{\Delta} + 1 \right)^2 \right)
&\geq \sqrt{\Delta} \left( 1 -  \frac{ n^{1/3} }{2^{2/3} n} \left( \frac{n^{1/3}}{2^{2/3}} + 1 \right)^2 \right)\\
&= \sqrt{\Delta} \left( 1 - \frac{1}{4} - \frac{1}{(2n)^{1/3}} - \frac{1}{(2n)^{2/3}} \right)\\
&>  \frac{2}{3}\sqrt{\Delta},
\end{align*}
where the last inequality follows since $\alpha \leq 1/7$ implies that $n \geq \binom{1/\alpha^2 - 1}{2} \geq 1128$. Thus we have established that $\sqrt{\Delta} < \frac{3( 1- \alpha)}{4 \alpha}$ and again using $\alpha \leq 1/7$, we therefore compute
\[
 \frac{\sqrt{\Delta}}{n} < \frac{3(1-\alpha)}{4\alpha} \frac{2\alpha^4}{(1-\alpha^2)(1-2\alpha^2)}
 = \frac{3 \alpha^3}{2(1+\alpha)(1-2\alpha^2)}
 < \frac{3 \alpha^3}{2}
\]
and
\[
 \Delta + 2 \sqrt{\Delta} + 1 < \frac{9 (1-\alpha)^2}{16 \alpha^2} + \frac{3(1-\alpha)}{2\alpha} + 1
 = \frac{9}{16 \alpha^2}  + \frac{3}{8 \alpha} +  \frac{1}{16}
 < \frac{2}{3 \alpha^2},
\]
The desired bound $\Delta < \frac{1}{4\alpha^2}$ now follows by applying \Cref{lem_beta_upper} and \Cref{lem_beta_lower} one final time with $t = \Delta$, yielding
\[
\frac{1 - \alpha}{2 \alpha} \geq \beta(G) 
\geq  \sqrt{\Delta} \left( 1 -  \frac{\sqrt{\Delta}}{n} \left(  \Delta + 2 \sqrt{\Delta} + 1  \right) \right)
> \sqrt{\Delta} (1-\alpha). \qedhere
\]
\end{proof}

It now remains to obtain one small final improvement to the bound of \Cref{lem_bootstrap} via another bootstrapping argument. To accomplish this, we derive a sharp Alon--Boppana-type theorem, under the assumption that the multiplicity of the second eigenvalue is large. To this end, we will need the following useful lemma. Note that a version of this lemma already appears in our recent work together with Buci\'c \cite{BB24} and is partially based on an idea appearing in \cite{JTYZZ21}.

\begin{lemma} \label{lem_optimal_multiplicity}
If $H$ is a non-empty subgraph of a graph $G$ with $\lambda_1(H) >  \beta(G)$, then 
\[
m_G(\beta(G)) \le \Delta_G |H|.
\]
\end{lemma}
\begin{proof}
First let $\beta = \beta(G)$, $\Delta = \Delta_G$ and observe that it is sufficient to prove the result for connected $H$. Indeed, one may always choose $H$ to be a vertex-minimal non-empty subgraph satisfying $\lambda_1(H) >  \beta$, which implies that $H$ is connected and therefore, there are at most $|H|+(\Delta-1)|H|=\Delta|H|$ vertices at a distance of at most one from $H$.  Observe that if we let $H'$ be the induced subgraph of $G$ consisting of vertices at a distance of at least two from $H$, then $H$ and $H'$ have disjoint vertex sets and there are no edges between them. 

We now claim that this implies $\lambda_1(H') < \beta$. Indeed, let $x \colon V(G) \to \R$ be the vector supported on $V(H)$ such that its restrictions to $V(H)$ is a unit eigenvectors for $\lambda_1(H)$, chosen via the Perron-Frobenius theorem to have only nonnegative entries. Similarly, let $x' \colon V(G) \to \R$ be such a vector for $H'$. Since $V(H)$ and $V(H')$ are disjoint, we have $\inprod{x}{x'} = 0$ and thus, there exist scalars $s,t \in \R$ such that $y = sx + tx'$ is orthogonal to $\one$ and $s^2 + t^2 = 1$ (so that $y$ is a unit vector). Furthermore, since the entries of $x$ and $x'$ are nonnegative, they are both not orthogonal to $\one$ and hence, $s$ and $t$ must be nonzero. Finally, since $H$ and $H'$ have no edges between them, we have $x^\intercal A x' = 0$ and thus,
\[
\beta \geq y^\intercal A y = s^2 x^\intercal A x + t^2 x'^\intercal A x' = s^2 \lambda_1(H) + (1 - s^2) \lambda_1(H') > s^2 \beta + (1-s^2) \lambda_1(H'),
\]
which gives $\lambda_1(H')  < \beta$, as desired.

We have now established that by starting with $G$ and removing at most $\Delta|H|$ vertices at a distance of at most one from $H$, we obtain a graph $H'$ with $\lambda_1(H') < \beta$, so $\m_{H'}(\beta) = 0$. The desired result then follows from the Cauchy interlacing theorem.
\end{proof}

\begin{lemma} \label{lem_bootstrap_optimal}
Let $0 < \alpha \leq 1/7$ and let $\Ce$ be a spherical $\{\alpha, -\alpha\}$-code in $\R^r$ with $|\Ce| \geq \max\left(\binom{1/\alpha^2 - 1}{2}, r + \frac{1}{16\alpha^4} +  \frac{1}{4\alpha^2} + 1 \right)$. If the largest eigenvalue $\lambda_1$ of the Gram matrix $M = M_\Ce$ satisfies $\lambda_1 > \frac{1-\alpha^2}{2\alpha^2}$ and has a corresponding eigenvector $x$ with no negative coordinates, then the maximum degree of the corresponding graph $G = G_{\Ce}$ satisfies
\[
\Delta_G \leq \left( \frac{1-\alpha}{2 \alpha} \right)^2. 
\]
\end{lemma}
\begin{proof}
Let $n = |\Ce|$ and let $\Delta = \Delta_G$ be the maximum degree of $G$ and let $A = A_G$ be the corresponding adjacency matrix. We first apply \Cref{lem_bootstrap} to obtain the bound $\Delta < \frac{1}{4\alpha^2}$. Since $n \geq r+2$, \Cref{lem_beta_upper} implies that $\beta = \beta(G) = \frac{1-\alpha}{2\alpha}$. 

Now suppose for sake of contradiction that $\Delta > \beta^2$ and let $H$ be a subgraph of $G$ consisting of a vertex and its $\Delta$ neighbors. \Cref{lem_star} implies that $\lambda_1(H) = \sqrt{\Delta} > \beta$ and thus, we may apply \Cref{lem_optimal_multiplicity} to conclude that $\m_G(\beta) \leq \Delta |H| < \Delta (\Delta + 1)$.
Using \eqref{eq_gram_adjacency}, the rank–nullity theorem, and the subadditivity of rank, one may conclude that $n - \m_G(\beta) \geq r + 1$ and hence,
\[
n \leq r + 1 + \Delta(\Delta + 1) < r + \frac{1}{16\alpha^4} +  \frac{1}{4\alpha^2} + 1,
\]
a contradiction.
\end{proof}

We now have all of the necessary ingredients to establish \Cref{thm_bound1_r}

\begin{proof}[Proof of \Cref{thm_bound1_r}] 
We first consider the case $\alpha > 1/7$. If $\alpha \notin \{1/3, 1/5\}$, then Neumann's bound gives $N^\R_\alpha(r) \leq 2r$ (see \cite{LS73}) . Also if $r < 1/\alpha^2 - 2$, then $N^\R_\alpha(r) < \binom{1/\alpha^2 - 1}{2}$ follows from the relative bound, so we may henceforth assume $r \geq 1/\alpha^2 - 2$. Otherwise if $\alpha \in \{1/3, 1/5\}$, then $N_\alpha^\R(r) \leq \binom{1/\alpha^2 -1}{2}$ follows from the theorems of Lemmens and Seidel \cite{LS73} and Cao, Koolen, Lin, and Yu \cite{CKLY22}, which were mentioned in \Cref{intro_real}. Therefore, we henceforth assume that $\alpha \leq 1/7$.

As explained in the outline, we let $\Ce$ be a spherical $\{\alpha, -\alpha\}$-code in $\R^r$ with $n = |\Ce| = N^\R_\alpha(r)$ given by \Cref{lem_nonnegative_eigenvector}, so that the Gram matrix $M = M_\Ce$ has largest eigenvalue $\lambda_1$ with a corresponding eigenvector $x$ having no negative coordinates. Since $2r - \frac{(1+\alpha)^2}{8 \alpha^2}$ is larger than $ \binom{1/\alpha^2 - 1}{2} $ precisely when $r > r_\alpha$, it will suffice to show that if $n \geq \binom{1/\alpha^2 - 1}{2}$, then exactly one of the following holds: Either $n \leq \binom{1/\alpha^2 - 1}{2}$ with equality only if $\Ce \subseteq \R^{1/\alpha^2 - 2}$ or otherwise $n \leq 2r - \frac{(1+\alpha)^2}{8 \alpha^2}$. These two cases will correspond to a dichotomy based on whether $\lambda_1$ is small or large.

If $\lambda_1 \leq \frac{1-\alpha^2}{2\alpha^2}$, then \Cref{lem_lambda_small} yields the first desired bound $n \leq \binom{1/\alpha^2 - 1}{2}$, as well as a characterization of equality. Otherwise, $\lambda_1 > \frac{1-\alpha^2}{2\alpha^2}$. Note that since $r \geq 1/\alpha^2 - 2$ and $\alpha \leq 1/7$, we have that $r + \frac{1}{16\alpha^4} +  \frac{1}{4\alpha^2} + 1 \leq 2r - \frac{(1+\alpha)^2}{8 \alpha^2}$ and therefore, if $n < r + \frac{1}{16\alpha^4} +  \frac{1}{4\alpha^2} + 1$ then we are done. Otherwise, $n \geq r + \frac{1}{16\alpha^4} +  \frac{1}{4\alpha^2} + 1$ and so, we may apply \Cref{lem_bootstrap_optimal} to conclude that the maximum degree of the corresponding graph $G = G_{\Ce}$ satisfies $\Delta_G \leq \left( \frac{1-\alpha}{2 \alpha} \right)^2$. 

Having established the required upper bound on $\Delta_G$, we observe that it implies a strong lower bound on $\lambda_1$. Indeed, via \eqref{eq_gram_adjacency}, we have that 
\[
\lambda_1 \geq \frac{\one^\intercal M \one}{n} 
\geq 1 - \alpha + \alpha n - 2 \alpha \Delta
\geq  \alpha n -  \frac{(1-\alpha)(1-3\alpha)}{2 \alpha}.
\]
It follows that the matrix $L = M - \lambda_1 x x^\intercal$ is close in Frobenius distance to the identity matrix $I$, so it remains to apply \Cref{lem_schnirelman} to obtain the desired upper bound on $n$. To this end, we compute that $\tr(L) = \tr(M) - \lambda_1 = n - \lambda_1$, $\rk(L) = r-1$, and using \eqref{eq_trace_square} that $\tr(L^2) = \lambda_2(M)^2 + \ldots + \lambda_r(M)^2 = \alpha^2 n^2 + (1-\alpha^2)n - \lambda_1^2$. Applying \Cref{lem_schnirelman} to $L$ yields
\[
f(\lambda_1) \coloneqq (r-1)(\alpha^2 n^2 + (1-\alpha^2)n - \lambda_1^2) - (n - \lambda_1)^2 = \tr(L^2)  \rk(L) - \tr(L)^2 \geq 0.
\]
By differentiating $f$, one can easily see that $f$ is decreasing if and only if $\lambda_1 \geq n/r$ and it follows from $n \geq \binom{1/\alpha^2 - 1}{2}$ and $r \geq 1/\alpha^2 - 2$ that we indeed have
$\lambda_1 \geq \alpha n - \frac{(1-\alpha)(1-3\alpha)}{2 \alpha} > n/r$. Therefore,
\begin{align*}
0 &\leq f\left(\alpha n - \frac{(1-\alpha)(1-3\alpha)}{2 \alpha} \right)\\
&= (r-1)\left(\alpha^2 n^2 + (1-\alpha^2)n - \left(\alpha n - \frac{(1-\alpha)(1-3\alpha)}{2 \alpha}  \right)^2 \right) - \left(n - \alpha n + \frac{(1-\alpha)(1-3\alpha)}{2 \alpha} \right)^2\\
&= (r-1)\left( 2(1-\alpha)^2 n  - \frac{(1-\alpha)^2(1-3\alpha)^2}{4 \alpha^2}  \right) - (1-\alpha)^2 \left(n + \frac{1-3\alpha}{2 \alpha} \right)^2\\
&= (1 - \alpha )^2 \left( -n^2 + 2\left( r - \frac{1 - \alpha}{2\alpha} \right) n - r \left( \frac{1-\alpha}{2 \alpha} - 1 \right)^2 \right).
\end{align*}
Hence, if we define $g(n) \coloneqq n^2 - 2\left( r - \frac{1 - \alpha}{2\alpha} \right) n +  r \left( \frac{1-\alpha}{2 \alpha} - 1 \right)^2$, then we have $g(n) \leq 0$ and $g$ is a quadratic polynomial with positive leading term and roots
\[
r - \frac{1 - \alpha}{2\alpha}  \pm \sqrt{\left( r - \frac{1 - \alpha}{2\alpha} \right)^2 -  r \left( \frac{1-\alpha}{2 \alpha} - 1 \right)^2}.
\]
It follows that $n$ is at most the largest root, i.e.\ we conclude that
\begin{align*}
n \leq r - \frac{1 - \alpha}{2\alpha}  + \sqrt{\left( r - \frac{1 - \alpha}{2\alpha} \right)^2 -  r \left( \frac{1-\alpha}{2 \alpha} - 1 \right)^2}
&\leq r - \frac{1 - \alpha}{2\alpha}  + r - \frac{1 - \alpha}{2\alpha}  -  r \frac{ \left( \frac{1-\alpha}{2 \alpha} - 1 \right)^2}{2 \left( r - \frac{1 - \alpha}{2\alpha}  \right)}\\
&\leq 2r - \frac{1 - \alpha}{\alpha} - \frac{1}{2}\left( \frac{1-\alpha}{2 \alpha} - 1 \right)^2\\
&= 2r - \frac{(1 + \alpha)^2}{8 \alpha^2}.
\qedhere
\end{align*}
\end{proof}

In the previous proof, we used \Cref{lem_beta_lower} with $H$ being a ball of radius 1, but in order to prove \Cref{thm_bound3_r}, we will take $H$ to be a ball of some radius $s \geq 2$. More specifically, for any graph $G$ with vertex $v \in V(G)$, we define $G(v,s)$ to be the subset of vertices for which there exists a path from $v$ of length at most $s$, i.e.\ the ball of radius $s$ centered at $v$ in $G$. We first recall the well-known fact that $|G(v,q)| \leq \sum_{i=0}^{s}{\Delta^i} = \frac{\Delta^{s+1} - 1}{\Delta - 1}$ where $\Delta = \Delta_G$ is the maximum degree. Indeed, this follows because there can be at most $\Delta^i$ vertices at a distance of exactly $i$ from $v$. We will also make use of the following lemma of Jiang \cite{J19} (see also Jiang and Polyanskii \cite{JP20}) showing that for any $s \in \N$, every graph has a ball of radius $s$ with a large eigenvalue.

\begin{lemma}[Jiang \cite{J19}] \label{lem_jiang}
Let $s \in \N$ and let $G$ be a graph on $n$ vertices with average degree $\overline{d} = \overline{d}(G) \geq 1$. Then there exists a vertex $v_0 \in V(G)$ such that 
\[
\lambda_1(G(v_0,s)) \geq 2 \sqrt{\overline{d}-1} \cos\left(\frac{\pi}{s+2}\right).
\]
\end{lemma}

\begin{proof}[Proof of \Cref{thm_bound3_r}]

Via \Cref{lem_nonnegative_eigenvector}, let $\Ce$ be a spherical $\{\alpha, -\alpha\}$-code in $\R^r$ with $n = |\Ce| = N^{\R}_{\alpha}(r)$ such that the Gram matrix $M = M_\Ce$ has a largest eigenvalue $\lambda_1$ with an eigenvector $x$ having no negative coordinates. As mentioned at the end of \Cref{intro_real}, by rotating each standard basis vector in $\R^r$ by the same angle towards the all ones vector $\one$, it is easy to see that we must have $n \geq r$. Moreover, by assumption we have $r \gg 1/\alpha^{2s + 1} \geq 1/\alpha^5$, so that \Cref{lem_lambda_small} implies that $\lambda_1 > \frac{1-\alpha^2}{2\alpha^2}$. It now follows from our previously established bounds that the average degree $\overline{d}$ and the maximum degree $\Delta$ of the corresponding graph satisfy $\overline{d} \leq \Delta \leq O(1/\alpha^2)$. Indeed, if $\alpha > 1/7$, then this follows from \Cref{lem_degree} and if $\alpha \leq 1/7$, then it follows from \Cref{lem_bootstrap}.

We now use \Cref{lem_friedman} in order to obtain a stronger bound on $\overline{d}$. If $\overline{d} \leq 1$ then we are done, and otherwise we may apply \Cref{lem_jiang} to obtain a vertex $v_0$ such that $H = G(v_0,s)$, i.e.\ the ball of radius $s$ centered at $v_0$, satisfies $\lambda_1(H) \geq 2 \sqrt{\overline{d} - 1} \cos\left( \frac{\pi}{s+2} \right)$. Also, if we let $y$ be a unit eigenvector corresponding to $\lambda_1(H)$, then by Cauchy--Schwarz we have $\inprod{y}{\one}^2 \leq ||y||^2 \cdot ||\one||^2 = |H| $ and since $|H| \leq \frac{\Delta^{s+1} - 1}{\Delta - 1} \leq O(\Delta^s)$, we have that
\[
\frac{\Delta \inprod{y}{\one}^2}{n} \leq O\left(\frac{\Delta^{s+1}}{r} \right) \leq O\left( \frac{1}{\alpha^{2s+2} r} \right) \leq \frac{o(1)}{\alpha}.
\]
Therefore, we may apply \Cref{lem_beta_upper} and \Cref{lem_friedman} in order to conclude
\[
\frac{1-\alpha}{2\alpha} \geq \beta(G) \geq \lambda_1(H) - \frac{2 \Delta \inprod{y}{\one}^2}{n} \geq 2\sqrt{\overline{d}-1} \cos\left( \frac{\pi}{s+2} \right) - \frac{o(1)}{\alpha},
\]
so that $\overline{d} \leq \left(\frac{1-\alpha}{2\alpha}\right)^2 \cdot \frac{1+o(1)}{4 \cos^2\left( \frac{\pi}{s+2} \right)} + 1$. The desired result now follows via \eqref{eq_simple_bound}.
\end{proof}

We now turn to proving \Cref{thm_bound4_r}, for which we will need the following sublinear bound on the multiplicity of the second eigenvalue of a connected graph obtained by Jiang, Tidor, Yao, Zhang, and Zhao \cite{JTYZZ21}.
\begin{theorem} \label{thm_JTYZZ}
There exists a constant $B > 0$ such that if $G$ is a connected graph on $n$ vertices with maximum degree $\Delta = \Delta(G)$, then the multiplicity of the second eigenvalue $\lambda_2 = \lambda_2(G)$ satisfies
\[
\m_G(\lambda_2) \leq \frac{B \log{\Delta}}{\log{\log{n}}} n.
\]
\end{theorem}
\begin{proof}
In the proof of Theorem 2.2 in \cite{JTYZZ21}, one may take $c(\Delta, 2) = \frac{1}{ B \log{\Delta} }$ for a sufficiently large constant $B$.
\end{proof}

For the lower bound, we will need the following lemma. However, since its proof, as well as that of \Cref{thm_bound4_r} follow directly from that of Proposition 3.2 and Theorem 1.2 of \cite{JTYZZ21}, respectively, we only provide sketches of the proofs here. In particular, we do not attempt to optimize the constant $C$ in \Cref{thm_bound4_r}.

\begin{lemma} \label{lem_lower_linear}
Let $0 < \alpha < 1$ be such that the spectral radius order $k = k\left(\frac{1-\alpha}{2\alpha}\right) < \infty$. Then for all $r \in \N$, we have
\[
N^\R_\alpha(r) \geq \left \lfloor \frac{k}{k-1} (r-1) \right \rfloor.
\]
\end{lemma}

\begin{proof}[Proof sketch of \Cref{lem_lower_linear}]
By the definition of spectral radius order, there exists a graph $H$ with $k$ vertices such that $\lambda_1(H) = \frac{1-\alpha}{2\alpha}$, so we may let $G$ consist of $\left \lfloor \frac{r-1}{k-1} \right \rfloor $ disjoint copies of $H$ together with $d-1 - (k-1)\left \lfloor \frac{r-1}{k-1} \right \rfloor$ isolated vertices. Note that $G$ has $n = \left \lfloor \frac{k (r-1)}{k-1} \right \rfloor$ vertices. Using the adjacency matrix $A = A(G)$ of $G$, one may define $M = (1-\alpha) I + \alpha J - 2 \alpha A$ and observe that it is an $n \times n$ positive semidefinite matrix with rank at most $r$, so that it is the Gram matrix of some spherical $\{\alpha, -\alpha\}$-code in $\R^r$ with $n$ elements. 
\end{proof}

\begin{proof}[Proof sketch of \Cref{thm_bound4_r}]
First note that the lower bound in item 1 of \Cref{thm_bound4_r} follows from \Cref{lem_lower_linear}. Now for the upper bounds, we begin as in previous proofs by using \Cref{lem_nonnegative_eigenvector} to obtain a $\Ce$ be a spherical $\{\alpha, -\alpha\}$-code in $\R^r$ with $n = |\Ce| = N^{\R}_{\alpha}(r)$ such that the Gram matrix $M = M_\Ce$ has a largest eigenvalue with an eigenvector $x$ having no negative coordinates. We choose $C$ sufficiently large so that $r \geq 2^{1/\alpha^C} > \binom{1/\alpha^2 - 1}{2}$, and it is easy to see that $n \geq r$, so we may apply \Cref{lem_lambda_small}, and \Cref{lem_degree} to conclude that the corresponding graph has maximum degree $\Delta < 1/\alpha^3$.

Since the nullspace of the corresponding Gram matrix $M = M_\Ce$ has dimension at least 2, using \eqref{eq_gram_adjacency} one may conclude that $\frac{1-\alpha}{2\alpha}$ is an eigenvalue of $G$. Therefore $\lambda_1(G) \geq \frac{1-\alpha}{2\alpha}$ and so if we consider the connected components $G_1, \ldots, G_t$ of $G$ ordered so that $\lambda_1(G_1) \geq \ldots \geq \lambda_1(G_t)$, then we either have $\lambda_1(G_1) > \frac{1-\alpha}{2\alpha}$ or $\lambda_1(G_1) = \frac{1-\alpha}{2\alpha}$.  

If $\lambda_1(G_1) > \frac{1-\alpha}{2\alpha}$, then one can show\footnote{Using \Cref{lem_beta_upper}, this follows from $\frac{1-\alpha}{2\alpha} \geq \beta(G) \geq z^\intercal A z$  where $z$ is taken to be a linear combination of eigenvectors corresponding to $\lambda_1(G_1)$ and $\lambda_1(G_i)$, such that $z \perp \one$ (see also the proof of \Cref{lem_optimal_multiplicity})} that $\lambda_1(G_i) < \frac{1-\alpha}{2\alpha}$ for all $i \geq 2$. Therefore, $ \m_{G_1}\left( \frac{1-\alpha}{2\alpha} \right) = \m_G\left( \frac{1-\alpha}{2\alpha} \right)$. Moreover, using \eqref{eq_gram_adjacency}, the rank--nullity theorem, and the subadditivity of rank, one may show that $n - \m_G\left( \frac{1-\alpha}{2\alpha} \right) \leq r + 1$ and so we may apply \Cref{thm_JTYZZ} to $G_1$ to conclude
\[
n - r - 1 \leq \m_{G}\left( \frac{1-\alpha}{2\alpha} \right)  =  \m_{G_1}\left( \frac{1-\alpha}{2\alpha} \right) \leq O \left( \frac{\log{\Delta}}{\log{\log{n}}}n \right) \leq O \left( \frac{\log(1/\alpha)}{\log{\log{r}}}n \right).
\] 
Since $r \geq 2^{1/\alpha^{C}}$, we obtain the bound $n \leq r + \frac{C \log(1/\alpha)}{2 \log{\log{r}}}r$ for $C$ sufficiently large.

Otherwise $\lambda_1(G_1) = \frac{1-\alpha}{2\alpha}$, in which case the existence of $G_1$ implies that the spectral radius $k = k\left( \frac{1-\alpha}{2\alpha} \right) < \infty$, and so we have proven claim 3 of \Cref{thm_bound4_r}. Furthermore, one can show that $\m_G\left( \frac{1-\alpha}{2\alpha} \right) \geq n - r + 1$ (Indeed, by the rank--nullity theorem, it suffices to verify\footnote{This follows from \eqref{eq_gram_adjacency} by using the fact that $(1-\alpha)I  - 2\alpha A$ and $\alpha J$ are both positive semidefinite (so that the intersection of their nullspaces is the nullspace of $M$), together with the Perron--Frobenius theorem (there exists an eigenvector of $A$ corresponding to $\frac{1-\alpha}{2\alpha}$ with all nonnegative coordinates, so that it is not in the nullspace of $\alpha J$).} that $\rk\left((1-\alpha)I - 2 \alpha A \right) \leq \rk(M) - 1$). Now, if we let $S = \left \{ i : \lambda_1(G_i) = \frac{1-\alpha}{2\alpha} \right \}$, then by definition of $k$ we have $|V(G_i)| \geq k$ for all $i \in S$ and hence $n \geq |S| k$. Moreover, the Perron--Frobenius theorem implies that $\lambda_2(G_i) < \lambda_1(G_i)$ for all $i$ and hence $|S| = \m_G\left( \frac{1-\alpha}{2\alpha} \right)$. Therefore, $n \geq  \m_G\left( \frac{1-\alpha}{2\alpha} \right) k \geq (n - r + 1)k$ and so we conclude that $n \leq \left \lfloor \frac{k (r-1)}{k-1} \right \rfloor$. 

Finally, observe that $r \geq 2^{1/\alpha^{C(k-1)}}$ implies $\left \lfloor \frac{k (r-1)}{k-1} \right \rfloor > r + \frac{C \log(1/\alpha)}{2 \log{\log{r}} }r$ for $C$ sufficiently large, which establishes claim 1 of \Cref{thm_bound4_r}. Otherwise if $r <  2^{1/\alpha^{C(k-1)}}$, then $\left \lfloor \frac{k (r-1)}{k-1} \right \rfloor < r + \frac{C \log(1/\alpha)}{\log{\log{r}} } r$ and so claim 2 of \Cref{thm_bound4_r} follows.
\end{proof}

\section{Equiangular lines in \texorpdfstring{$\C^r$}{C\^{}r}} \label{section_c}

The goal of this section will be to prove \Cref{thm_bound1_c}. The argument we will present is a generalization of part of the proof of \Cref{thm_bound1_r}, with \Cref{main_r} replaced by \Cref{main_c}. The key difference will be that we do not use a bootstrapping\footnote{Although one can generalize the definition of degree to the complex setting, we don't know of an appropriate Alon--Boppana-type theorem which would work for this definition.} argument and instead directly use \Cref{main_c} in order to obtain bounds on the largest eigenvalue of the corresponding Gram matrix, see the outline given below.

Now let $r \in \N$ and let $\Le$ be a set of $n$ equiangular lines in $\C^r$. As previously noted, if we choose a complex unit vector along each line $\ell \in \Le$, the resulting collection $\Ce$ forms a spherical $S^1(\alpha)$-code, where $S^1(\alpha)$ is the circle of radius $\alpha$ centered at $0$ in $\C$. Similar to \Cref{section_r}, we say that $\Ce$ \emph{represents} $\Le$ and so instead of working with $\Le$ directly, we will consider some spherical $S^1(\alpha)$-code $\Ce$ which represents $\Le$. Just as in \Cref{section_r}, the case $\alpha = 0$ is trivial and so we assume $\alpha > 0$. 

\begin{proof}[Outline of the proof of \Cref{thm_bound1_c}]
We will begin by summing the inequalities given by \Cref{main_c} over all $u \in \Ce$ in order to show that the largest eigenvalue $\lambda_1 = \lambda_1(M)$ of the corresponding Gram matrix $M = M_\Ce$ satisfies $p(\lambda_1) \geq 0$ where $p$ is a degree 3 polynomial with a positive leading term and 3 real roots $s_0 \geq s_1 > s_2$. It will follow that $\lambda_1$ is either at most $s_1$ or at least $s_0$ and we further show that $s_1 \leq \frac{1-\alpha^2}{\alpha^2}$ and $s_0 \geq \alpha n - \frac{1 + O(\alpha)}{2\alpha^2}$ when $n$ is sufficiently large. Therefore, just like in the proof of \Cref{thm_bound1_r}, we have two cases: either $\lambda_1$ is small (at most $\frac{1-\alpha^2}{\alpha^2}$) or large (nearly $\alpha n$)

If $\lambda_1$ is small, i.e. $\lambda_1 \leq \frac{1-\alpha^2}{\alpha^2}$, then using a straightforward generalization of \Cref{lem_lambda_small}, we will conclude that  $n \leq \left( \frac{1}{\alpha^2} - 1 \right)^2$. Otherwise, $\lambda_1$ is large, i.e.\ $\lambda_1 \geq \alpha n - \frac{1 + O(\alpha)}{2\alpha^2}$, in which case we will proceed as in the proof of \Cref{thm_bound1_r} by applying \Cref{lem_schnirelman} with $L = M - \lambda_1 x x^*$, where $x$ is a unit eigenvector corresponding to $\lambda_1$, in order to conclude that 
\[
n \leq \frac{1 + O(\alpha)}{\alpha} r. \qedhere
\]
\end{proof}

We first establish the required generalization of \Cref{lem_lambda_small}.

\begin{lemma} \label{lem_lambda_small_c}
Let $0 < \alpha < 1$ and let $\Ce$ be a spherical $S^1(\alpha)$-code. If the largest eigenvalue $\lambda_1$ of the Gram matrix $M = M_\Ce$ satisfies $\lambda_1 \leq \frac{1-\alpha^2}{\alpha^2}$, then
\[
|\Ce| \leq \left( \frac{1}{\alpha^2} - 1 \right)^2
\]
with equality only if the span of $\Ce$ has dimension $1/\alpha^2 - 1$.
\end{lemma}
\begin{proof}
As in the proof of \Cref{lem_lambda_small}, if we let $n = |\Ce|$ and $r = \rk(M)$, then we can observe that \eqref{eq_trace_square} and \eqref{eq_lambda_small} also hold for the Hermitian Gram matrix $M$, so we have
\[
(\alpha^2n + 1 - \alpha^2) n \leq \lambda_1 n  \leq  \frac{1-\alpha^2}{\alpha^2} n ,
\]
from which the desired bound follows. Moreover, it follows from the inequality \eqref{eq_lambda_small} that equality occurs in the above only if $\lambda_i(M) = \frac{1-\alpha^2}{\alpha^2}$ for all $1 \leq i \leq r$. By changing bases, we may assume that $\Ce \subset \C^{r}$, so that if we let $V \in \C^{r \times \Ce}$ be the matrix given by $Vy = \sum_{v \in \Ce}{y(v) v}$, then $M = V^* V$ has the same nonzero eigenvalues as $V V^* = \sum_{v \in \Ce}{v v^*}$. Therefore, $n = \left( 1/\alpha^2 - 1 \right)^2$ implies that $ \sum_{v \in \Ce}{v v^*} = \frac{1-\alpha^2}{\alpha^2} I$ where $I$ is the $r \times r$ identity matrix and by taking the trace, we conclude that $\left( 1/\alpha^2 - 1 \right)^2 = n = \frac{1-\alpha^2}{\alpha^2} r$ and so $r = 1/\alpha^2 - 1$.
\end{proof}

We will also need the following simple upper bound on the largest eigenvalue of the corresponding Gram matrix.

\begin{lemma} \label{lem_lambda1}
Let $0 < \alpha < 1$ and let $\Ce$ be a spherical $\{ \alpha, -\alpha \}$-code in $\C^r$ with corresponding Gram matrix $M = M_{\Ce}$. Then $\lambda_1(M) \leq \alpha |\Ce| + 1 - \alpha$.
\end{lemma}
\begin{proof}
Let $x$ be a unit eigenvector corresponding to $\lambda_1(M)$. Using Cauchy--Schwarz, we obtain
\begin{align*}
\lambda_1(M) = x^\intercal M x 
\leq \sum_{u, v \in \Ce}{ |M(u,v)| |x(u)| |x(v)|}
&= (1 - \alpha)  \sum_{v \in \Ce}{|x(v)|^2} + \alpha \left(  \sum_{v \in \Ce}{|x(v)|} \right)^2\\
&\leq (1 - \alpha)  \sum_{v \in \Ce}{|x(v)|^2} + \alpha |\Ce| \sum_{v \in \Ce}{|x(v)|^2} \\
&= 1 - \alpha + \alpha |\Ce|. \qedhere
\end{align*}
\end{proof}

We now obtain the aforementioned degree 3 polynomial inequality as a function of the largest eigenvalue of the corresponding Gram matrix.

\begin{lemma} \label{lem_eigen_poly_c} 
Let $0 < \alpha < 1$ and let $\Ce$ be a spherical $S^1(\alpha)$-code in $\C^r$ with $n = |\Ce|$ whose Gram matrix $M = M_{\Ce}$ has largest eigenvalue $\lambda_1 = \lambda_1(M)$. Then we have
\[
\lambda_1^3 - \alpha^2 \left( n + \frac{1 - \alpha^2}{\alpha^2} \right)^2 \lambda_1 + (1-\alpha^2) \left(n + \frac{1-\alpha^2}{\alpha^2} \right) n \geq 0,
\]
with equality if $n = r^2$.
\end{lemma}
\begin{proof}
Let $x$ be a unit eigenvector corresponding to $\lambda_1$. For each $u \in \Ce$, \Cref{main_c} yields
\[
\lambda \left(\frac{\lambda^2}{ \alpha^2 n + 1 - \alpha^2} - \frac{1-\alpha^2}{\alpha^2} \right) |x(u)|^2 \geq \lambda - \frac{1-\alpha^2}{\alpha^2}.
\]
Since $\alpha, \lambda_1 > 0$ and $x$ is a unit vector, summing the above over all $u \in \Ce$ and multiplying by $\alpha^2 n + 1 - \alpha^2$ yields the desired inequality. Moreover, if $n = r^2$, then we have equality via \Cref{main_c}.
\end{proof}

In order to bound the largest root of the degree 3 polynomial in the preceding lemma, we need the following computational lemma.
\begin{lemma} \label{lem_largest_root}
Let $p(x) = x^3 - b x + c$ be a polynomial such that $b > 0$, $c \geq 0$, and $ \frac{c}{b^{3/2}} \leq  \frac{2}{3 \sqrt{3}}$. Then the largest root of $p$ is at least $\sqrt{b} - \frac{c}{2b} - \left( \frac{27}{4} - 3 \sqrt{3} \right)\frac{ c^2}{b^{5/2} }$.
\end{lemma}
\begin{proof}
Let $q =  \frac{27}{4} - 3 \sqrt{3}$, $z = \frac{c}{b^{3/2}}$, and $x = 1 - \frac{z}{2} - q z^2$ so that $\sqrt{b} x  =  \sqrt{b} - \frac{c}{2b} - q \frac{c^2}{ b^{5/2} }$. Since the leading term of $p$ is positive, we need to show that $p \left( \sqrt{b} x \right) \leq 0$. To this end, we observe that the polynomial $g(z) =  \frac{3}{4} - 2q + \left(3q - \frac{1}{8} \right) z +  \left(3q^2 - \frac{3q}{4} \right)  z^2 - \frac{3q^2}{2} z^3 - q^3 z^4$ satisfies 
\begin{align*}
\frac{p\left( \sqrt{b} x \right)}{b^{3/2}} 
= x^3 - x + z
= z^2 g(z),
\end{align*}
so we equivalently need to show that $g(z) \leq 0$ for all $0 \leq z \leq \frac{2}{3 \sqrt{3}}$. To verify this fact, we compute\footnote{Although these computations can all be done by hand, we remark that the reader can also convince themselves of their validity by making use of a mathematical software such as Matlab, Mathematica, or WolframAlpha.} that $g\left(\frac{2}{3 \sqrt{3}}\right) = 0$, so it suffices to check that $g'(z) > 0$ for all $0 \leq z \leq \frac{2}{3 \sqrt{3}}$. To see this, we compute that $g'(0) > 0$, $g'(-1) < 0$, $g'\left( \frac{2}{3 \sqrt{3}} \right) > 0$ and note that $g'$ is a degree 3 polynomial with a negative leading term, so that it must have roots $r_0, r_1, r_2$ satisfying $r_2 < -1 < r_1 < 0$ and $r_0 > \frac{2}{3 \sqrt{3}}$. 
\end{proof}

\begin{remark}
Note that the discriminant of $x^3 - bx + c$ is $\frac{c^2}{4} - \frac{b^3}{27}$, which is at most 0 precisely when $\frac{c}{b^{3/2}} \leq \frac{2}{3 \sqrt{3}}$, in which case Vi\`{e}te's formula implies that the largest root of this polynomial is $2 \sqrt{\frac{b}{3}} \cos\left( \frac{1}{3}  \arccos\left( \frac{- 3 \sqrt{3} c}{2 b^{3/2}} \right) \right)$. However, to make use of this expression, we would still need to approximate it by a Taylor polynomial with an explicit bound on the error term, so we found it simpler to derive such an approximation directly. Also, the coefficient $\frac{27}{4} - 3 \sqrt{3}$ is best possible.
\end{remark}

We now obtain the desired bounds on the roots of the polynomial from \Cref{lem_eigen_poly_c}. In particular, it will be convenient to obtain a bound on the square of the largest root, as follows.

\begin{lemma} \label{lem_root_bounds_c}
Let $0 < \alpha < 1$ and $n \in \N$ be such that $n \geq \left( \frac{1}{\alpha^2} - 1 \right)^2$. Then the polynomial $p(x) = x^3 - \alpha^2 \left( n + \frac{1 - \alpha^2}{\alpha^2} \right)^2 x + (1-\alpha^2)\left(n + \frac{1-\alpha^2}{\alpha^2} \right) n$ has 3 real roots $s_0 \geq s_1 \geq s_2$ satisfying $s_1 \leq \frac{1-\alpha^2}{\alpha^2}$ and $s_0^2 > n(\alpha^2 n + 1 - \alpha^2)  - \frac{ (1-\alpha^2)(1-\alpha)}{ \alpha} n -  3 \frac{(1-\alpha^2)^2}{ \alpha^4}$.
\end{lemma}
\begin{proof}
Let $b = \alpha^2 \left( n + \frac{1 - \alpha^2}{\alpha^2} \right)^2$ and $c =  (1-\alpha^2) \left(n + \frac{1-\alpha^2}{\alpha^2} \right) n$, so that $p(x) = x^3 - b x + c$. We would like to apply \Cref{lem_largest_root}, so we need to bound $\frac{c}{b^{3/2}} = \frac{ (1-\alpha^2) n }{\alpha^3 \left( n + \frac{1-\alpha^2}{\alpha^2} \right)^2}$. One can verify by differentiating this expression with respect to $n$, that it is decreasing when $n \geq \frac{1 - \alpha^2}{\alpha^2}$ and we have $n \geq \max\left( \left( \frac{1}{\alpha^2} - 1 \right)^2, 1 \right) \geq \frac{1 - \alpha^2}{\alpha^2}$, so $\frac{c}{b^{3/2}}$ is maximized when $n = \max\left( \left( \frac{1}{\alpha^2} - 1 \right)^2, 1 \right)$. Furthermore, observe that for both $n = \left( \frac{1}{\alpha^2} - 1 \right)^2$ and $n = 1$, we have $\frac{c}{b^{3/2}} = \alpha(1-\alpha^2)$, so that we may conclude 
\[
\frac{c}{b^{3/2}}  \leq \alpha(1-\alpha^2) \leq \frac{2}{3\sqrt{3}},
\] 
where the last inequality holds since $\alpha(1-\alpha^2)$ is maximized when $\alpha = 1/\sqrt{3}$. We may now apply \Cref{lem_largest_root} to conclude that $s_0 \geq \sqrt{b} - \frac{c}{2b} -  \left( \frac{27}{4} - 3 \sqrt{3} \right) \frac{c^2}{b^{5/2} }$ and since $2 \left( \frac{27}{4} - 3 \sqrt{3} \right) - 1/4 < 3$, we obtain the desired bound on the largest root
\begin{align*}
s_0^2 > b - \frac{c}{\sqrt{b}} - 3\frac{ c^2}{ b^2}
&= n(\alpha^2 n + 1 - \alpha^2)  + (1-\alpha^2) n +  \frac{\left(1-\alpha^2\right)^2}{\alpha^2} - \frac{ 1-\alpha^2 }{ \alpha} n - 3\frac{(1-\alpha^2)^2 n^2}{ \alpha^4 \left(n + \frac{1-\alpha^2}{2\alpha^2} \right)^2}\\
&> n(\alpha^2 n + 1 - \alpha^2)  - \frac{ (1-\alpha^2)(1-\alpha)}{ \alpha} n -  3\frac{(1-\alpha^2)^2}{ \alpha^4}.
\end{align*}

Now we compute that $p(0) > 0$ and since $n \geq \left( \frac{1}{\alpha^2} -1 \right)^2$, we also compute that $p\left(\frac{1-\alpha^2}{\alpha^2}\right) \leq 0$. Therefore, as $p$ has a positive leading term, its smallest root must satisfy $s_2 < 0$ and its largest root must satisfy $s_0 \geq \frac{1-\alpha^2}{\alpha^2}$. Hence, either $s_0 = \frac{1-\alpha^2}{\alpha^2}$ or else $p$ has a root in the interval $\left[ 0, \frac{1-\alpha^2}{\alpha^2} \right]$ which is not $s_0$ or $s_2$, so it must be $s_1$. In either case, we conclude the other desired bound $s_1 \leq \frac{1-\alpha^2}{\alpha^2}$.
\end{proof}

\begin{proof}[Proof of \Cref{thm_bound1_c}]
We let $\Ce$ be a spherical $S^1(\alpha)$-code representing a set of $n = N^{\C}_{\alpha}(r)$ equiangular lines in $\C^r$ and let $M = M_{\Ce}$ be the corresponding Gram matrix with largest eigenvalue $\lambda_1 = \lambda_1(M)$. \Cref{lem_eigen_poly_c} implies that $p(\lambda_1) = \lambda_1^3 - \alpha^2 \left( n + \frac{1 - \alpha^2}{\alpha^2} \right)^2 \lambda_1 + (1-\alpha^2)\left(n + \frac{1-\alpha^2}{\alpha^2} \right) n \geq 0$. Just like the proof of \Cref{thm_bound1_r}, we will show that exactly one of the following holds: either $n \leq (1/\alpha^2 - 1)^2$ with equality only if $\Ce$ forms a SIC in $\C^{1/\alpha^2 - 1}$ or $n \leq \frac{1+\alpha}{\alpha} r + 3\frac{1+\alpha}{\alpha^3}$. Thus we may assume without loss of generality that $n \geq \left( \frac{1}{\alpha^2} - 1\right)^2$ and therefore apply \Cref{lem_root_bounds_c} to conclude that $p$ has real roots $s_0 > s_1 > s_2$ satisfying $s_1 \leq \frac{1-\alpha^2}{\alpha^2}$ and $s_0^2 > n(\alpha^2 n + 1 - \alpha^2)  - \frac{ (1-\alpha^2)(1-\alpha)}{ \alpha} n -  3\frac{(1-\alpha^2)^2}{\alpha^4} $. Since $p$ has a positive leading term, $p(\lambda_1) \geq 0$ implies that either $\lambda_1 \leq s_1$ or $\lambda_1 \geq s_0$. If $\lambda_1 \leq s_1 \leq \frac{1-\alpha^2}{\alpha^2} $, then the first desired bound $n \leq \left(1/\alpha^2 - 1 \right)^2$ and a characterization of equality follow from \Cref{lem_lambda_small_c}. 

Otherwise, we have $\lambda_1 \geq s_0$. Now let $x$ be a unit eigenvector corresponding to $\lambda_1$ and let $L = M - \lambda_1 x x^*$. We compute that $\rk(L) = r-1$ and using \Cref{lem_lambda1}, that $\tr(L) = n - \lambda_1 \geq (1-\alpha)(n-1)$. Moreover, using \eqref{eq_trace_square} (which also holds for the Hermitian Gram matrix $M$), we have
\begin{align*}
\tr(L^2) = \lambda_2(M)^2 + \ldots + \lambda_r(M)^2 = n(\alpha^2 n + 1 - \alpha^2) - \lambda_1^2
&\leq n(\alpha^2 n + 1 - \alpha^2) - s_0^2\\
&< \frac{ (1-\alpha^2)(1-\alpha)}{\alpha} n +  3\frac{(1-\alpha^2)^2 }{\alpha^4} 
\end{align*}
Therefore,  \Cref{lem_schnirelman} yields
\[
(1-\alpha)^2 (n-1)^2 
\leq  \tr(L)^2 
\leq (r-1)  \tr(L^2)
< (r-1) \left( \frac{ (1-\alpha^2)(1-\alpha)}{\alpha} n +  3\frac{(1-\alpha^2)^2 }{\alpha^4} \right)
\]
and dividing by $(1-\alpha)^2 n$, we conclude that
\begin{equation*} \label{eq_spectral_bound_c}
n  - 2 < \frac{(n-1)^2}{n} < \frac{1+\alpha}{\alpha}(r-1)  +  3\frac{(1+\alpha)^2 }{\alpha^4} \frac{r}{n} . 
\end{equation*}
The second desired bound $n \leq \frac{1+\alpha}{\alpha} r + 3\frac{1+\alpha}{\alpha^3}$ now follows from the above inequality. Indeed, we may assume that $n \geq \frac{1+\alpha}{\alpha}r$ as otherwise we are done, and so we have
\[
n < \frac{1+\alpha}{\alpha} r - \frac{1 + \alpha}{\alpha} + 2 + 3\frac{(1+\alpha)^2 }{\alpha^4} \frac{r}{n}
\leq \frac{1+\alpha}{\alpha} r +  3\frac{1+\alpha }{\alpha^3}. \qedhere
\]
\end{proof}

\section{Eigenvalues of regular graphs} \label{section_graphs}

In this subsection, we will prove \Cref{regular_bounds} and \Cref{cor_alon_boppana}, which can be seen as generalizations of the Alon--Boppana theorem to dense graphs. Our approach will be to use \eqref{eq_gram_adjacency} in order to convert the given regular graph into a corresponding system of real equiangular lines and then apply the projection method of \Cref{section_projections}. As a consequence, our new bounds will be tight for any strongly regular graph corresponding to $\binom{r+1}{2}$ equiangular lines in $\R^r$.  To this end, we first recall some well-known facts about the second eigenvalue of connected graphs, see e.g.\ \cite{CDS98} p.\ 163.

\begin{lemma} \label{lem_spectral_history}
$\lambda_2(K_n) = -1$ for any complete graph $K_n$, $\lambda_2(K_{n_1, \ldots, n_r}) = 0$ for any complete multipartite graph $K_{n_1, \ldots, n_r}$, and $\lambda_2(H) > 0$ for any other nontrivial connected graph $H$.
\end{lemma}

\begin{remark}
Beyond \Cref{lem_spectral_history}, there is a line of research characterizing graphs $H$ with $\lambda_2(H) \leq c$ or $\lambda_2(H) = c$ for small constant $c$, see \cite{CS95}. There are also important results on characterizing graphs $H$ with the smallest eigenvalue $\lambda_n(H) \geq - c$ for small $c$, such as the classical result for $c = 2$ \cite{CGSS91, GR01} and more recent generalizations \cite{KYY18, JP21}.
\end{remark}

We now derive an improved version of \Cref{main_r} in the special case where $y = \one$ is an eigenvector of $M$. 

\begin{theorem} \label{main_r_regular}
Let $0 < \alpha < 1$ and let $\Ce$ be a spherical $\{ \alpha, -\alpha\}$-code in $\R^r$ with corresponding Gram matrix $M = M_{\Ce}$. If $n = |\Ce|$ and $\one$ is an eigenvector of $M$ with corresponding eigenvalue $\lambda \neq 0$, then
\[
\frac{(1-\alpha^2)n}{2 \lambda} \inprod{x}{Mx} + \left( \frac{\lambda^2}{n} - \frac{1 - \alpha^2}{2} \right) \inprod{x}{\one}^2 \geq \inprod{x}{M^2 x}
\]
for all $x \in \R^{\Ce}$, with equality whenever $n = \binom{r+1}{2} - 1$.
\end{theorem}
\begin{proof}
Let us suppose we are in the setting of the proof of \Cref{main_r}, so that
$\Ce = \{v_1, \ldots, v_n\}$, the linear map $\W \colon \R^{\Ce} \rightarrow \Se_r$ is defined by $\W e_v = v v^\intercal$ for $v \in \Ce$ and the orthogonal projection $\Pe \colon \Se_r \rightarrow \Se_r$ onto the span of $\{ v v^\intercal : v \in \Ce\}$ is given by $\Pe = \W (\W^{\#} \W)^{-1} \W^{\#}$. 

Now let $x \in \R^\Ce$ and define $X_x = \frac{1}{2} \left( Vx (V\one)^{\intercal} + V\one (Vx)^{\intercal}\right)$. We compute that
\begin{equation} \label{y_frobenius}
||X_x||_F^2 = \frac{1}{2}\left(\inprod{x}{Mx} \inprod{\one}{M \one} + \inprod{x}{M\one}^2 \right)
= \frac{1}{2}\left(\lambda n \inprod{x}{Mx} + \lambda^2 \inprod{x}{\one}^2 \right).
\end{equation}
Moreover, for all $v \in \Ce$, we have $\inprod{v v^\intercal}{V\one(Vx)^{\intercal}}_F 
= \inprod{e_v}{M \one} \inprod{e_v}{Mx}
= \lambda (Mx)(v)$ and similarly $\inprod{v v^\intercal}{Vx(V\one)^{\intercal}}_F = \lambda (Mx)(v)$, so that $\W^\# X_x = \lambda Mx$. As in the proof of \Cref{main_r}, we have $(\W^{\#} \W)^{-1} = \frac{1}{1-\alpha^2}\left(I - \frac{\alpha^2}{\alpha^2 n + 1 - \alpha^2} J \right)$ and so 
\begin{align} \label{py_frobenius}
||\Pe X_x||_F^2 
= X_x^\# \W (\W^\# \W)^{-1} \W^\# X_x
&= \frac{\lambda^2}{1-\alpha^2} (Mx)^\intercal \left( I - \frac{\alpha^2}{\alpha^2 n + 1 - \alpha^2} J \right) M x  \\
&= \frac{\lambda^2}{1-\alpha^2}\left(\left | \left | Mx \right | \right |^2 - \frac{\alpha^2 \lambda^2}{\alpha^2 n + 1 - \alpha^2}  \inprod{x}{ \one}^2 \right).\nonumber
\end{align}
In particular, applying \eqref{y_frobenius} and \eqref{py_frobenius} with $x = \one$ yields $||X_{\one}||_F^2 = \lambda^2 n^2$ and
\[
||\Pe X_{\one}||_F^2 = \frac{1}{1-\alpha^2}\left(\lambda^4 n - \frac{\alpha^2 \lambda^4 n^2}{\alpha^2 n + 1 - \alpha^2} \right)
= \frac{\lambda^4 n }{\alpha^2 n + 1 - \alpha^2}.
\]
We now claim that $X_{\one} - \Pe X_{\one} \neq 0$. Indeed, otherwise we would have 
\[
0 = ||X_{\one} - \Pe X_{\one}||_F^2 = ||X_{\one}||_F^2 - ||\Pe X_{\one}||_F^2 = \lambda^2 n \left(n - \frac{\lambda^2}{\alpha^2 n + 1 - \alpha^2} \right),
\]
which implies $\lambda^2 = n ( \alpha^2 n + 1 - \alpha^2)$. However, we also have via \eqref{eq_trace_square} that $\sum_{i=1}^n{\lambda_i(M)^2} = n (\alpha^2 n + 1 - \alpha^2)$, so that all eigenvalues of $M$ except $\lambda$ must be $0$ and thus $M = \frac{\lambda}{n} \one \one^{\intercal}$, which contradicts the fact that $M(v,v) = 1$ and $|M(u,v)| = \alpha < 1$ for any $u \neq v \in \Ce$. 

Therefore, we may define 
\[
Z = \frac{X_{\one} - \Pe X_{\one}}{\sqrt{||X_{\one}||_F^2 - ||\Pe X_{\one}||_F^2 }} = \frac{X_{\one} - \Pe X_{\one}}{\lambda \sqrt{n^2 - \frac{\lambda^2 n}{\alpha^2n + 1 - \alpha^2} }}
\]
so that by definition, $Z$ is orthogonal to $v v^\intercal$ for all $v \in \Ce$. It follows that $\Pe' = \Pe + Z Z^{\#}$ is the orthogonal projection onto the span of $\{v v^\intercal : v \in \Ce\} \cup Z$ and moreover, 
\begin{equation} \label{eq_projection_graph}
||X_x||_F^2 \geq ||\Pe' X_x||_F^2 = ||\Pe X_x||_F^2 + \inprod{Z}{X_x}_F^2.
\end{equation}
In view of \eqref{y_frobenius} and \eqref{py_frobenius}, it remains to compute $\inprod{Z}{X_x}_F$. To this end, observe that
\[
\inprod{X_{\one}}{Vx(V\one)^{\intercal}}_F = \inprod{V \one}{Vx} \inprod{V \one}{V \one} 
= \inprod{M \one}{x} \inprod{\one}{M \one}
= \lambda^2 n \inprod{x}{\one}
\] 
and similarly $\inprod{X_{\one}}{V\one(Vx)^{\intercal}}_F = \lambda^2 n \inprod{x}{\one}$, so that $\inprod{X_{\one}}{X_x}_F = \lambda^2 n \inprod{x}{\one}$. Since $\W^\# X_{\one} = \lambda M\one = \lambda^2 \one$, we conclude that 
\begin{align*}
\inprod{\Pe X_{\one}}{X_x} 
= X_{\one}^{\#} \W (\W^{\#} \W)^{-1} \W^{\#} X_x
&= \frac{\lambda^3}{1 - \alpha^2} \one^{\intercal} \left(I - \frac{\alpha^2}{\alpha^2 n + 1 - \alpha^2} J \right) Mx\\
&= \frac{\lambda^3}{1 - \alpha^2} \left(1 - \frac{\alpha^2 n}{\alpha^2 n + 1 - \alpha^2} \right) \one^{\intercal} Mx \\
&= \frac{\lambda^4}{\alpha^2 n + 1 - \alpha^2} \inprod{x}{\one}
\end{align*}
and so we have
\begin{align*}
\inprod{Z}{X_x}_F 
= \frac{\inprod{X_{\one}}{X_x}_F - \inprod{\Pe X_{\one}}{X_x}_F}{\lambda \sqrt{n^2 - \frac{\lambda^2 n}{\alpha^2n + 1 - \alpha^2} }}
&= \frac{\lambda^2 n \inprod{x}{\one} - \frac{\lambda^4}{\alpha^2 n + 1 - \alpha^2} \inprod{x}{\one}}{\lambda \sqrt{n^2 - \frac{\lambda^2 n}{\alpha^2n + 1 - \alpha^2} }}\\
&= \frac{\lambda}{n} \sqrt{n^2 - \frac{\lambda^2 n }{\alpha^2 n + 1 - \alpha^2}} \inprod{x}{\one}.
\end{align*}
Now that we have determined $\inprod{Z}{X_x}_F$, we conclude via \eqref{y_frobenius}, \eqref{py_frobenius}, and \eqref{eq_projection_graph} that
\[
\frac{ \lambda n}{2} \inprod{x}{Mx} \geq \frac{\lambda^2}{1-\alpha^2}\inprod{x}{M^2 x} - \left( \frac{\lambda^4}{(1-\alpha^2) n} - \frac{\lambda^2}{2} \right) \inprod{x}{\one}^2,
\]
which is equivalent to the desired bound. 

Moreover, note that $\Se_r$ has dimension $\binom{r+1}{2}$ so that if $n = \binom{r+1}{2} - 1$, then since $\rk(\W) = n$ and $Z$ is orthogonal to $v v^\intercal$ for all $v \in \Ce$, we have that $\{v v^\intercal : v \in \Ce\} \cup Z$ span $\Se_r$ and thus $\Pe'$ is the identity map, giving equality above.
\end{proof}

\begin{remark}
The inequality of \Cref{main_r_regular} is equivalent to the matrix $\frac{1-\alpha^2}{2\lambda} M + \left( \frac{\lambda^2}{n} - \frac{1 - \alpha^2}{2} \right) J - M^2$ being positive semidefinite.
\end{remark}


\begin{lemma} \label{lem_regular_bounds_equi}
Let $0 < \alpha < 1$ and let $\Ce$ be a spherical $\{ \alpha, -\alpha\}$-code in $\R^r$ with corresponding Gram matrix $M = M_{\Ce}$. If $n = |\Ce|$ and $\one$ is an eigenvector of $M$ with corresponding eigenvalue $\lambda \neq 0$, then
\[
\lambda^2 \geq n(\alpha^2 n + 1 - \alpha^2) - \frac{1-\alpha^2}{2} n \left( \frac{n}{\lambda} - 1 \right),
\]
and for any eigenvalue $\mu$ of $M$ with an eigenvector orthogonal to $\one$, we have
\[
\mu \leq \frac{(1-\alpha^2)n}{2 \lambda},
\]
with equality in both whenever $n = \binom{r+1}{2} - 1$.
\end{lemma}
\begin{proof}
Fix any $v \in \Ce$ and observe that $\inprod{e_v}{M^2 e_v} = \alpha^2 n + 1 - \alpha^2$. Thus applying \Cref{main_r_regular} with $x = e_v$, we obtain
\[
\frac{1-\alpha^2}{2} \left( \frac{n}{\lambda} - 1 \right) + \frac{\lambda^2}{n} = \frac{(1-\alpha^2)n}{2 \lambda} + \frac{\lambda^2}{n} - \frac{1 - \alpha^2}{2} \geq \alpha^2 n + 1 - \alpha^2,
\]
which is equivalent to the first bound. 

For the second bound, if $\mu = 0$ then it holds trivially. Otherwise, we let $x$ be a unit eigenvector for $\mu$. Since $\inprod{x}{\one} = 0$, we have via \Cref{main_r_regular} that $\frac{(1-\alpha^2)n}{2 \lambda} \mu \geq \mu^2$, so dividing by $\mu$ gives the desired bound. Moreover, $n = \binom{r+1}{2} - 1$ implies we have equality in \Cref{main_r_regular} and therefore also in both of the above bounds.
\end{proof}

Using \Cref{lem_regular_bounds_equi}, we are now able to obtain the desired bounds on the eigenvalues of regular graphs and prove \Cref{regular_bounds}.

\begin{proof}[Proof of \Cref{regular_bounds}]
First note that if $k = 0$ then both inequalities hold trivially. Assume now that $k \geq 1$. We claim that $\lambda_2 > 0$ follows from \Cref{lem_spectral_history}. Indeed, otherwise $G$ would have to either be a complete graph, a complete multipartite graph, or disconnected. Moreover, $G$ cannot be the complete graph or a regular complete multipartite graph, since they all have spectral gap at least $n/2$. Otherwise, $G$ isn't connected, in which case the Perron--Frobenius theorem implies $\lambda_2 = k \geq 1$.

Now let $A = A(G)$ be the adjacency matrix of $G$ and for each $i \in [n]$, let $\lambda_i = \lambda_i(G)$ be the $i$th largest eigenvalue of $G$. By the spectral theorem for symmetric matrices, we have $A = \sum_{i = 1}^{n}{\lambda_i u_i u_i^{\intercal}}$ for an orthonormal basis of eigenvectors $u_1, \ldots, u_n \in \R^n$. As previously mentioned, the Perron--Frobenius theorem implies $\lambda_1 = k$ and $u_1$ can be taken to be a scalar multiple of $\one$. Since $||\one|| = \sqrt{n}$ we may take $u_1 = \frac{1}{\sqrt{n}} \one$. Since $J = \one \one^{\intercal}$, we therefore have
\[ A = \frac{k}{n} J + \sum_{i = 2}^{n}{\lambda_i u_i u_i^{\intercal}}. \]
Let $U = (u_1, \ldots, u_n)$ be the $n \times n$ matrix with $u_1, \ldots, u_n$ as columns. Since $u_1, \ldots, u_n$ form an orthonormal basis, $U$ is an orthogonal matrix and so we have
\[ 
I = U U^{\intercal} = \sum_{i=1}^{n}{u_i u_i^{\intercal}} = \frac{1}{n} J + \sum_{i=2}^{n}{u_i u_i^{\intercal}}.
\]
We now define $\alpha = \frac{1}{2\lambda_2 + 1}$ and note that $0 < \alpha < 1$ since $\lambda_2 > 0$. In view of \eqref{eq_gram_adjacency}, we define $M = \alpha J - 2 \alpha A + (1-\alpha)I$ and observe that $M(v,v) = 1$ and $| M(u,v) | = \alpha$ for all $u \neq v  \in \Ce$. Moreover, we have that
\begin{align*}
M 
= \frac{1}{2 \lambda_2 + 1} \left( J - 2A + 2\lambda_2 I \right)
= \frac{1}{2 \lambda_2 + 1} \left( \left(1 - \frac{2(k - \lambda_2)}{n} \right)J + 2\sum_{i=2}^n{(\lambda_2 - \lambda_i)} u_i u_i^{\intercal} \right),
\end{align*}
which implies that $M$ is positive semidefinite and has $\one$ as an eigenvector with corresponding eigenvalue $\lambda = \frac{n - 2(k-\lambda_2)}{2 \lambda_2 + 1} > 0$. It follows that the dimension of the null space of $M$ is precisely the multiplicity of $\lambda_2$ in $G$, $m_G(\lambda_2)$, so that $r = \rk(M) = n - m_G(\lambda_2)$. Therefore, we can associate a vector in $\R^r$ to each vertex of $G$, so that the resulting collection $\Ce$ has Gram matrix $M_\Ce = M$. It follows that $\Ce$ is a spherical $\{\alpha, -\alpha\}$-code and so we may apply \Cref{lem_regular_bounds_equi} to obtain
\begin{align*}
\lambda^2
&\geq n(\alpha^2 n + 1 - \alpha^2) - \frac{(1 - \alpha^2)n^2}{2 \lambda} + \frac{(1 - \alpha^2)n }{2}\\
&= \frac{n(n + (2 \lambda_2 + 1)^2 - 1)}{(2 \lambda_2 + 1)^2} - \frac{((2 \lambda_2 + 1)^2 - 1)n^2}{2 (2 \lambda_2 + 1) (n - 2(k-\lambda_2))} + \frac{(2\lambda_2 + 1)^2 - 1) n}{2(2 \lambda_2 + 1)^2}.
\end{align*}
Multiplying by $\frac{(2\lambda_2 + 1)^2}{n}$, we therefore have
\begin{align*}
n - 4k + 4\lambda_2 + \frac{4(k-\lambda_2)^2}{n}
= \frac{(2\lambda_2 + 1)^2}{n} \lambda^2
\geq n - \frac{2 \lambda_2(2 \lambda_2 + 1)( \lambda_2 + 1)}{ 1 - \frac{2(k-\lambda_2)}{n}} + 6 \lambda_2^2 + 6\lambda_2,
\end{align*}
which implies the first desired bound
\[
2k - \frac{2(k-\lambda_2)^2}{n} \leq  \frac{ \lambda_2(2 \lambda_2 + 1)( \lambda_2 + 1)}{ 1 - \frac{2(k-\lambda_2)}{n}} - \lambda_2 (3 \lambda_2 + 1).
\]
Moreover, it follows from the spectral decomposition of $M$ that it also has $u_n$ as an eigenvector with corresponding eigenvalue $\mu = \frac{2(\lambda_2 - \lambda_n)}{2 \lambda_2 + 1}$ and so we may apply  \Cref{lem_regular_bounds_equi} to obtain
\[
\frac{2(\lambda_2 - \lambda_n)}{2 \lambda_2 + 1}
= \mu 
\leq \frac{(1 - \alpha^2)n}{2 \lambda}
= \frac{((2\lambda_2 + 1)^2 - 1)n}{2 (2\lambda_2 + 1) (n - 2(k-\lambda_2))}.
\]
Multiplying by $\frac{(2 \lambda_2 + 1)}{2}$ and subtracting $\lambda_2$ we conclude the second desired bound
\[
-\lambda_n \leq \frac{ \lambda_2 (\lambda_2 + 1)}{1 - \frac{2(k-\lambda_2)}{n}} - \lambda_2.
\]
Finally, $n = \binom{n - m(\lambda_2)+1}{2} - 1 = \binom{r+1}{2} - 1$ implies we have equality in both bounds of \Cref{lem_regular_bounds_equi} and therefore also in both of the above.
\end{proof}

\begin{proof}[Proof of \Cref{cor_alon_boppana}]
We first suppose that there exists a constant $\varepsilon > 0$ such that $k - \lambda_2 \leq (1 - \varepsilon)\frac{n}{2}$, or equivalently $1 - \frac{2(k - \lambda_2)}{n} \geq \varepsilon$. Via \Cref{regular_bounds} and \Cref{rem_regular} we have
\[
k < 2 \left( k - \frac{(k-\lambda_2)^2}{n } \right) 
\leq  \frac{ \lambda_2(2 \lambda_2 + 1)( \lambda_2 + 1)}{ \varepsilon} - \lambda_2 (3 \lambda_2 + 1) 
\leq O(\lambda_2^3),
\]
and therefore $\lambda_2 \geq \Omega \left( k^{1/3} \right)$. Also via \Cref{regular_bounds}, we have $-\lambda_n \leq \frac{ \lambda_2 (\lambda_2 + 1)}{\varepsilon} - \lambda_2 \leq O(\lambda_2^2)$, so that $\lambda_2 \geq \Omega\left(\sqrt{-\lambda_n}\right)$. Moreover, if $G$ is bipartite, then $\lambda_n = -k$ and so $\lambda_2 \geq \Omega \left( \sqrt{k} \right)$.

Now let us further suppose that $k - \lambda_2 = o(n)$. In this case, \Cref{regular_bounds} implies
\begin{align*}
(1-o(1))2k + o(1) \lambda_2
= 2 \left( k - \frac{(k-\lambda_2)^2}{n } \right)
\leq  \frac{ \lambda_2(2 \lambda_2 + 1)( \lambda_2 + 1)}{ 1 - o(1)} - \lambda_2 (3 \lambda_2 + 1)
&= (1 + o(1)) 2 \lambda_2^3,
\end{align*}
so that $\lambda_2 \geq (1 - o(1)) k^{1/3}$. Again using \Cref{regular_bounds} we have $-\lambda_n \leq \frac{ \lambda_2 (\lambda_2 + 1)}{1 - o(1)} - \lambda_2 \leq (1 + o(1) ) \lambda_2^2$, so that $\lambda_2 \geq (1- o(1)) \sqrt{-\lambda_n}$. Finally, if $G$ is bipartite, we have $\lambda_n = -k$ and so $\lambda_2 \geq (1-o(1)) \sqrt{k}$.
\end{proof}

\section{Concluding remarks} \label{concluding_remarks}

In this section, we make some concluding remarks and suggest some directions for future research. Firstly, we provide a table that summarizes the consequences of our new upper bounds \Cref{thm_bound1_r},  \Cref{thm_bound3_r}, and \Cref{thm_bound4_r}, together with the relative bound. In the following, $k = k\left(\frac{1-\alpha}{2\alpha}\right)$ is the corresponding spectral radius, $C$ is a positive constant, and all of the asymptotic bounds are assuming $r \rightarrow \infty$ and $\alpha \rightarrow 0$.
\[
N^\R_{\alpha}(r) \leq 
\begin{cases}
	\frac{1 - \alpha^2}{1 - \alpha^2 r} r & \text{if } r \leq  \frac{1}{\alpha^2} - 2 \\[4pt]
	\binom{1/\alpha^2 -1}{2} & \text{if } \frac{1}{\alpha^2} - 2 < r \leq \frac{1 - o(1)}{4\alpha^4}\\[4pt]
	2r - \frac{(1+\alpha)^2}{8 \alpha^2}  & \text{if }  \frac{1-o(1)}{4\alpha^4} < r \leq O\left(  \frac{1}{\alpha^{5}} \right)\\[4pt]
    	\left(1 + \frac{1 + o(1)}{4 \cos^2 \left(\frac{\pi}{q+2}\right)} \right) r  & \text{if }  \frac{1}{\alpha^{2q+1}} \ll r \leq O\left(  \frac{1}{\alpha^{2q+3}} \right) \text{ for integer } q \geq 2\\[4pt]
    	\left( \frac{5}{4}  + o(1) \right) r & \text{if } 1/\alpha^{\omega(1)} \leq r < 2^{1/\alpha^{4C}}\\[4pt]
    	\left(1 + \frac{C \log(1/\alpha)}{\log{\log{r}}} \right) r & \text{if } 2^{1/\alpha^{4C}} \leq r < 2^{1/\alpha^{C (k-1)}}\\[4pt]
    \left \lfloor \frac{k}{k-1}(r-1) \right \rfloor & \text{if } 2^{1/\alpha^{C (k-1)}} \leq r.
\end{cases}
\]
Any significant improvements to our upper bounds or novel lower bound constructions would be interesting. In the following, we present several promising directions for further research.

\begin{enumerate}
 
\item Regarding real equiangular lines, one of the most interesting questions is to determine how far the approach of Jiang, Tidor, Yao, Zhang, and Zhao \cite{JTYZZ21} can be extended. Their method relied on two ingredients: A bound on the maximum degree $\Delta$ of a corresponding graph $G$ which only depends on $\alpha$, and a sublinear bound on the multiplicity $m(\lambda_2)$ of the second eigenvalue $\lambda_2$ of any connected graph with maximum degree at most $\Delta$. Since we now have strong bounds on $\Delta$, the only limitation to extending their result is that their bound on second eigenvalue multiplicity $m(\lambda_2)$ for a connected graph with maximum degree $\Delta$ is only $O\left(n \log{\Delta} / \log{\log{n}} \right)$. Note that for normalized adjacency matrices and therefore also for regular graphs, McKenzie, Rasmussen, and Srivastava \cite{MRS21} improve on the approach of \cite{JTYZZ21} to improve the upper bound to $O_{\Delta}(n/\log^{c}{n})$ for some constant $c$. However, they give reason to suggest that their methods cannot go beyond $O_{\Delta}\left(n / \log{n} \right)$. Moreover, Haiman, Schildkraut, Zhang, and Zhao \cite{HSZZ22} give a construction of a graph with max degree $\Delta = 4$ and $m(\lambda_2) \geq \sqrt{n / \log{n}}$, and they also point out that for bounded degree ($\Delta = O(1)$) graphs, $\sqrt{n}$ is a natural barrier for group representation based constructions such as theirs. 
    
 \item Regarding the multiplicity of the second eigenvalue of a graph, we also point out that a related question was considered by Colin de Verdi\`{e}re \cite{C87}. He conjectured that the maximum of the multiplicity of the second smallest eigenvalue of a generalized Laplacian operator, over all such operators on an orientable surfaces $S$, is precisely $\chr(S) - 1$ where $\chr(S)$ is the chromatic number of the surface. Moreover, via Heawood's formula \cite{RY68} it is known that $\chr(S) = \left \lfloor \left (7 + \sqrt{48\text{g}(S) + 1} \right)/2 \right \rfloor$ where $\text{g}(S)$ is the genus of $S$. In view of this, Tlusty \cite{T207} formulates an analogous statement in the setting of graphs: Given a graph $G$, the maximum multiplicity of the second eigenvalue of a weighted\footnote{For a graph $G$ with vertex set $V(G) = [n]$, a symmetric matrix $L \in \R^{n \times n}$ is called a weighted Laplacian of $G$ if $L \one = 0$ and for all $i \neq j$, $L_{i,j} = 0$ when $ij \notin E(G)$ and $L_{i,j} < 0$ when $ij \in E(G)$.} Laplacian matrix on $G$ is $\left \lfloor \left (7 + \sqrt{48\text{g}(G) + 1} \right)/2 \right \rfloor - 1$, where $g(G)$ is the minimal genus of a surface in which $G$ embeds. Moreover, he proves this claim for a class of graphs including paths, cycles, complete graphs, and their Cartesian products. It is not hard\footnote{Embed the vertices $V(G)$ arbitrarily on the sphere and for each edge $uv \in E(G)$, add a handle near $u$ and $v$ and pass the edge along this handle.} to see that $g(G) \leq |E(G)|$ and thus if Tlusty's statement were to be true for say, the unweighted Laplacian $k I - A$ of a $k$-regular graph $G$ on $n$ vertices with adjacency matrix $A$, then we would be able to conclude that $m_{G}(\lambda_2) \leq O \left( \sqrt{|E(G)|} \right) = O\left(\sqrt{k n} \right)$. Also note that for graphs whose corresponding matrix satisfies a certain technical condition known as the strong Arnold hypothesis, Theorem 5 of Pendavingh \cite{P98} implies the asymptotic version of this result $m_{G}(\lambda_2) \leq O(\sqrt{|E(G)|})$. In particular, for bounded degree graphs this give a bound of $O(\sqrt{n})$, matching the natural barrier discussed previously.
 
\item As mentioned in \Cref{rem_direct_lambda}, it would be interesting to see if some of the results in this paper can be reinterpreted in terms of the Lasserre hierarchy, in order to verify a conjecture of de Laat, Keizer, and Machado \cite{LKM22}.
    
\item From the point of view of applications, one of the most important questions regarding complex equiangular is Zauner's conjecture, i.e.\ that there exist $r^2$ equiangular lines (a SIC) in $\C^r$ for all $r$. An interesting consequence of \Cref{thm_bound1_c} is that if one constructs $r^2$ equiangular lines in $\C^{(1-o(1))(r+1)^{3/2}}$ with $\alpha = 1/\sqrt{r+1}$, then one immediately obtains a SIC in $\C^r$, thereby providing a potential alternative approach to proving Zauner's conjecture.

\item We note that the construction of $r^2 - r + 1$ equiangular lines in $\C^r$ due to Godsil and Roy \cite{GR09} is \emph{flat}, i.e.\ the coordinates of the unit vectors which span the lines all have the same magnitude. They also showed that their construction is best possible, i.e.\ there are at most $r^2 - r + 1$ flat equiangular lines in $\C^r$. It would, therefore, be interesting to see if our methods can be used to obtain an exact version of \Cref{cor_godsil_roy} for flat equiangular lines.
    
\item Another natural question is whether the graph-based methods used in the real case can be extended to complex equiangular lines. We remark that starting with the Gram matrix $M$ of a spherical $S^1(\alpha)$-code $\Ce$ in $\C^r$, one may use \eqref{eq_gram_adjacency} to define a matrix $A$ which is like a complex generalization of an adjacency matrix, and moreover that our approach can be generalized to give an upper bound on $\max_{v \in \Ce}{\text{Re}\left(e_v^{\intercal} A {\one}\right)}$, a quantity analogous to the maximum degree. However, we do not know of an appropriate Alon--Boppana-type theorem in this setting.
    
\item Since we have a way of generalizing the maximum degree bound to complex equiangular lines, it would be interesting to determine if the results of Jiang, Tidor, Yao, Zhang, and Zhao \cite{JTYZZ21} can be extended in order to exactly determine $N^{\C}_{\alpha}(r)$ when $r$ is large relative to $1/\alpha$. We note that this may require complex versions of the bound on the multiplicity of the second largest eigenvalue, as well as a complex version of the Perron--Frobenius theorem.
    
\item We observe that the projection method used in this paper can be applied to other matrices associated with a graph, such as the unweighted Laplacian, in order to obtain new spectral inequalities. Indeed, the method applies for any positive semidefinite matrix and any matrix can be made as such by adding a sufficient multiple of the identity.
    
\item It would be interesting to obtain generalizations of our results to arbitrary spherical $L$-codes in $\R^r$ with $|L| = s$, i.e.\ to $s$-distance sets for $s \geq 2$, as well as for $L = [-1, \alpha]$, i.e.\ the classical coding theory question which is equivalent to packing spherical caps on a sphere in $\R^r$. In particular, as a first step towards understanding $[-1, \alpha]$-codes, it would be interesting to see if our methods can be applied for $[-\alpha, \alpha]$-codes.
    
\item We expect that our approach should extend beyond lines to higher-order equiangular subspaces with respect to different notions of angle, which are described in \cite{BS19}. In particular, we predict that our methods can be generalized to equiangular subspaces with respect to the chordal distance.
    
\item The first Welch bound is the first in a family of higher-order Welch bounds. By generalizing our projection method to higher-order tensors, we expect that it is possible to obtain an improvement to all of the Welch bounds in the same way as we have done for the first.
    
\item Note that spherical $\{\alpha, -\alpha\}$-codes correspond to signed complete graphs, and so it would be interesting to generalize and apply our methods to signed graphs and more generally, to unitary signings of graphs. In particular, Koolen, Cao, and Yang \cite{KCY21} have recently used a Ramsey-theoretic approach analogous to \cite{BDKS18} in order to study signed graphs, so it is plausible that our methods can be applied to extend their result.
    
\end{enumerate}

\vspace{0.4cm}
\noindent
{\bf Acknowledgments.}\,
We thank Noga Alon and Michael Krivelevich for stimulating discussions, as well as Dustin Mixon for information regarding the novelty of one of our bounds. We also thank Boris Bukh, Yufei Zhao, and Gaston Burrull for helpful feedback on earlier versions of this paper. Finally, we thank the anonymous referees for their careful reading and insightful suggestions, which have helped improve this paper's exposition.

\end{document}